\documentclass[11pt,twoside]{amsart}

\DeclareRobustCommand{\SkipTocEntry}[4]{}
\makeatletter
\def\@tocline#1#2#3#4#5#6#7{\relax
  \ifnum #1>\c@tocdepth 
  \else
    \par \addpenalty\@secpenalty\addvspace{#2}%
    \begingroup \hyphenpenalty\@M
    \@ifempty{#4}{%
      \@tempdima\csname r@tocindent\number#1\endcsname\relax
    }{%
      \@tempdima#4\relax
    }%
    \parindent\z@ \leftskip#3\relax \advance\leftskip\@tempdima\relax
    \rightskip\@pnumwidth plus4em \parfillskip-\@pnumwidth
    #5\leavevmode\hskip-\@tempdima
      \ifcase #1
       \or\or \hskip 1em \or \hskip 2em \else \hskip 3em \fi%
      #6\nobreak\relax
    \dotfill\hbox to\@pnumwidth{\@tocpagenum{#7}}\par
    \nobreak
    \endgroup
  \fi}
\makeatother

\makeatletter
\def\subsubsection{\@startsection{subsubsection}{3}%
  \z@{.5\linespacing\@plus.7\linespacing}{-.5em}%
  {\normalfont\bfseries}}
 \makeatother
 
\usepackage[all]{xy}
\usepackage{amsmath}
\usepackage{amsfonts}
\usepackage{amssymb}
\usepackage{amsthm}
\usepackage{color}
\usepackage{enumerate}
\usepackage{hyperref}
\usepackage{fullpage}
\usepackage{graphicx}
\usepackage{url}

\theoremstyle{definition}
\newtheorem{defn}{Definition}[section]

\newtheorem{question}[defn]{Question}
\theoremstyle{plain}
\newtheorem{thm}[defn]{Theorem}

\newtheorem{lem}[defn]{Lemma}
\newtheorem{prop}[defn]{Proposition}
\newtheorem{cor}[defn]{Corollary}

\newtheorem{conj}[defn]{Conjecture}

\newtheorem{assuInt}{Assumption}[section]

\def\D{\ensuremath{\mathbb{D}}}
\def\F{\ensuremath{\mathbb{F}}}

\def\P{\ensuremath{\mathbb{P}}}

\def\R{\ensuremath{\mathbb{R}}}
\def\Z{\ensuremath{\mathbb{Z}}}

\def\FF{\ensuremath{\mathcal F}}

\def\HH{\ensuremath{\mathcal H}}
\def\II{\ensuremath{\mathcal I}}

\def\OO{\ensuremath{\mathcal O}}

\def\TT{\ensuremath{\mathcal T}}

\def\ch{\mathop{\mathrm{ch}}\nolimits}

\def\Coh{\mathop{\mathrm{Coh}}\nolimits}

\def\Db{\mathop{\mathrm{D}^{\mathrm{b}}}\nolimits}

\def\dim{\mathop{\mathrm{dim}}\nolimits}

\def\ext{\mathop{\mathrm{ext}}\nolimits}
\def\Ext{\mathop{\mathrm{Ext}}\nolimits}

\def\Hom{\mathop{\mathrm{Hom}}\nolimits}

\def\RlHom{\mathop{\mathbf{R}\mathcal Hom}\nolimits}
\def\RHom{\mathop{\mathbf{R}\mathrm{Hom}}\nolimits}

\def\mod{\mathop{\mathrm{mod}}\nolimits}
\def\min{\mathop{\mathrm{min}}\nolimits}

\def\into{\ensuremath{\hookrightarrow}}
\def\onto{\ensuremath{\twoheadrightarrow}}

\begin{document}

\title{Derived categories and the genus of space curves}

\author{Emanuele Macr\`i}
\address{Northeastern University, Department of Mathematics, 360 Huntington Avenue, Boston, MA 02115-5000, USA}
\email{e.macri@northeastern.edu}
\urladdr{https://web.northeastern.edu/emacri/}

\author{Benjamin Schmidt}
\address{The University of Texas at Austin, Department of Mathematics, 2515 Speedway, RLM 8.100, Austin, TX 78712, USA}
\email{schmidt@math.utexas.edu}
\urladdr{https://sites.google.com/site/benjaminschmidtmath/}



\begin{abstract}
We generalize a classical result about the genus of curves in projective space by Gruson and Peskine to principally polarized abelian threefolds of Picard rank one. The proof is based on wall-crossing techniques for ideal sheaves of curves in the derived category. In the process, we obtain bounds for Chern characters of other stable objects such as rank two sheaves. The argument gives a proof for projective space as well. In this case these techniques also indicate an approach for a conjecture by Hartshorne and Hirschowitz and we prove first steps towards it.
\end{abstract}

\maketitle

\tableofcontents


\section{Introduction}
\label{sec:intro}

A celebrated result in the theory of space curves is the following (\cite{Hal82:genus_space_curves, GP78:genre_courbesI, Har80:joe_space_curves}).

\begin{thm}[Gruson--Peskine, Harris]
\label{thm:GP}
Let $d,k>0$ and $g\geq0$ be integers.
Let $C\subset \P^3$ be an integral curve of degree $d$ and arithmetic genus $g$.
Assume:
\begin{itemize}
\item $H^0(\P^3,I_C(k-1))=0$, and
\item $d>k(k-1)$.
\end{itemize}
Then
\[
g \leq \frac{d^2}{2k} + \frac{1}{2}d(k-4) + 1 - \varepsilon,
\]
for
\[
\varepsilon = \frac{1}{2}f\left(k-f-1+\frac{f}{k}\right),
\]
where $d \equiv -f (\mod k)$ and $0 \leq f < k$.
\end{thm}

For example, if $k=1$ this says that the largest genus for a fixed degree is given by that of a plane curve, i.e., $g\leq \frac{(d-1)(d-2)}{2}$. For $k=2$ it corresponds to Castelnuovo's inequality for non-planar curves \cite[(IV, 6.4)]{Har77:algebraic_geometry}. The first goal of this article is to prove a version of this theorem for other threefolds by using the theory of stability in the derived category. The second goal is to attack a conjecture by Hartshorne and Hirschowitz for $d \leq k(k-1)$ in the case of $\P^3$ with similar techniques.

\subsubsection*{Tilt Stability}

Given a curve $C\subset\P^3$, there are two exact sequences in the category of coherent sheaves associated to its ideal sheaf $I_C$. For a non-zero section of $H^0(\P^3,I_C(h))$, we can simply consider the associated sequence
\[
0 \to \OO_{\P^3}(-h) \to I_C \to F \to 0.
\]
Otherwise, for a non-zero section of $H^2(\P^3,I_C(m-4)) = \Ext^1(\II_C, \OO_{\P^3}(-m))$, we can consider the corresponding extension
\[
0 \to \OO_{\P^3}(-m) \to E \to I_C \to 0.
\]
The genus of $C$ can be bounded, by bounding the Chern characters of both $E$ and $F$. Just using the first exact sequence is not enough to conclude the proof of Theorem \ref{thm:GP}, since the bound therein is not decreasing for $h$ large. The key observation in our approach is that $E$ can also be thought of as a subobject of the ideal sheaf $I_C$, but in a different abelian category. Using a notion of stability on these categories, and by just taking the first factor of the Harder-Narasimhan filtration of $I_C$ with respect to this stability, we can select a canonical sequence among all these.

More generally, let $X$ be a smooth projective threefold. We are going to use the notion of \emph{tilt-stability}. It is reviewed in Section \ref{sec:background}. This is a weak stability condition in the bounded derived category of coherent sheaves on $X$, which was introduced in \cite{BMT14:stability_threefolds} (based on Bridgeland stability on surfaces \cite{Bri08:stability_k3,AB13:k_trivial}). It can be thought of as a generalization of the classical notion of slope stability for sheaves on surfaces. If we fix an ample divisor $H$ on $X$, it roughly amounts to replacing the category of coherent sheaves with the heart of a bounded t-structure $\Coh^\beta(X)$ in the bounded derived category $\Db(X)$ and the classical slope with a new slope function $\nu_{\alpha,\beta}$. Everything depends on two real parameters $\alpha, \beta \in \R$, $\alpha > 0$. The starting point for us is that, for $\alpha \gg 0$ and $\beta < 0$, and for any curve $C \subset X$, the ideal sheaf $\II_C$ is $\nu_{\alpha,\beta}$-stable (see Lemma \ref{lem:large_volume_limit} for details). The key idea is to study variation of stability for $\II_C$ with to respect to $\alpha$ and $\beta$.


\subsubsection*{The main theorem}

Let $X$ be a smooth projective threefold of Picard rank one, i.e., its N\'eron-Severi group is generated by the class of a single ample divisor $H$. For a subvariety $Y \subset X$ of dimension $n= 1, 2$, we define its \emph{degree} as $H^n\cdot Y/H^3$. We also define an extension of the remainder term in Theorem \ref{thm:GP} as follows. For a rational number $d\in \frac{1}{2}\Z$ and an integer $k \geq 1$, we set
\[
\varepsilon(d, k) = \frac{1}{2}f\left(k-f-1+\frac{f}{k}\right) + \varepsilon(d, 1) \text{ for }
\varepsilon(d, 1) = \begin{cases}
\frac{1}{24} &\text{if $d \notin \Z$} \\
0 &\text{if $d \in \Z$,}
\end{cases}
\]
where $d \equiv -f (\mod k)$ and $0 \leq f < k$, $f\in\frac{1}{2}\Z$.

\begin{thm}[See Theorem \ref{thm:genus_bound_large_degree}]
\label{thm:genus_bound_large_degree_intro}
Assume $X$ satisfies Assumptions \ref{assu:div_curves}, \ref{assu:classic_bg}, \ref{assu:bmt}. Let $k \in \Z_{>0}$ and $d \in \frac{1}{2}\Z_{>0}$, and let $C \subset X$ be an integral curve of degree $d$. Further, assume
\begin{itemize}
\item $H^0(X, I_C((k-1)H')=0$ for any divisor $H'$ in the same numerical class as $H$, and
\item $d > k(k-1)$.
\end{itemize}
Then
\[
\frac{\ch_3(\II_C)}{H^3} \leq E(d,k) := \frac{d^2}{2k} + \frac{dk}{2} - \varepsilon(d,k).
\]
\end{thm}

For example, in characteristic $0$ the assumptions of Theorem \ref{thm:genus_bound_large_degree_intro} are satisfied in the case of $\P^3$, principally polarized abelian threefolds of Picard rank one, and index $2$ Fano threefolds of Picard rank one with degree one or two. In fact, the case of $\P^3$ is independent of the characteristic of the field, and Theorem \ref{thm:GP} holds in that case. We note that already the case $k = 1$ strengthens a conjecture by Debarre \cite[Section 5]{Deb94:degree_curves} and a result by Pareschi--Popa \cite[Theorem B]{PP08:Castelnuovo_theory} for the special case of principally polarized abelian threefolds of Picard rank one (see \cite{LN16:theta_regularity_curves} for results on general polarized abelian varieties).
The precise assumptions are the following.

\begin{assuInt}
\label{assu:div_curves}
The N\'eron-Severi group is generated by the class of an ample divisor $H$. Moreover, the Chern character of any $E \in \Coh(X)$ satisfies $\ch_2(E) \in \tfrac{1}{2} H^2 \cdot \Z$ and $\ch_3(E) \in \tfrac{1}{6} H^3 \cdot \Z$.
\end{assuInt}

Without the part about $\ch_2(E)$ in Assumption \ref{assu:div_curves} curves of small degree pose issues. Some bounds can be proved without this assumption, but they do not seem optimal. However, the Picard rank one assumption is more important in our argument, since otherwise determining the tilt stability of ideal sheaves becomes substantially more involved. The part about $\ch_3(E)$ is for computational reasons to bound the Chern characters of rank two sheaves.

\begin{assuInt}
\label{assu:classic_bg}
Any slope semistable sheaf $E \in \Coh(X)$ satisfies
\[
\Delta(E) := \frac{(H^2 \cdot \ch_1(E))^2 - 2(H^3 \cdot \ch_0(E))(H \cdot \ch_2(E))}{(H^3)^2} \geq 0.
\]
\end{assuInt}

Assumption \ref{assu:classic_bg} is well known to be true in characteristic $0$, being a consequence of the classical Bogomolov inequality (\cite{Rei78:bogomolov,Bog78:inequality,Gie79:bogomolov}). In positive characteristic it is only sometimes satisfied, for example for $\P^3$ and abelian threefolds (see \cite{Lang04:positive_char} for more details). 

For any $\beta \in \R$ and $E \in \Db(X)$ we define the twisted Chern character $\ch^{\beta}(E) := \ch(E) \cdot e^{-\beta H}$. Note that for $\beta \in \Z$ this is simply saying $\ch^{\beta}(E) = \ch(E(-\beta H))$.

\begin{assuInt}
\label{assu:bmt}
For any $\nu_{\alpha,\beta}$-semistable object $E \in \Coh^{\beta}(X)$ the inequality
\[
Q_{\alpha, \beta}(E) := \alpha^2 \Delta(E) + \frac{4(H \cdot \ch_2^{\beta}(E))^2}{(H^3)^2} - \frac{6 (H^2\cdot \ch_1^{\beta}(E)) \ch_3^{\beta}(E)}{(H^3)^2} \geq 0
\]
holds.
\end{assuInt}

Assumption \ref{assu:bmt} is the crucial ingredient in the proof. It roughly tells us that the ideal sheaf of a curve of large genus has to be destabilized at a certain point, and it allows us to reduce the number of possible walls. This assumption is part of a more general conjecture in \cite{BMT14:stability_threefolds,BMS16:abelian_threefolds} for characteristic $0$. The case of $\P^3$ was shown in \cite{Mac14:conjecture_p3} and the proof actually works in any characteristic. The smooth quadric hypersurface in $\P^4$ was done in \cite{Sch14:conjecture_quadric}. Later, both of these were generalized to Fano threefolds of Picard rank one in \cite{Li15:conjecture_fano_threefold}. Moreover, the case of abelian threefolds were handled independently in \cite{MP15:conjecture_abelian_threefoldsI,MP16:conjecture_abelian_threefoldsII} and \cite{BMS16:abelian_threefolds}, and Calabi-Yau threefolds of abelian type in \cite{BMS16:abelian_threefolds}. Most recently, it was shown in \cite{Kos17:stability_products} for the case of $\P^2 \times E$, $\P^1 \times \P^1 \times E$, and $\P^1 \times A$, where $E$ is an arbitrary elliptic curve, and $A$ is an arbitrary abelian surface. For higher Picard rank it is known to be false in general. Counterexamples were given in \cite{Sch17:counterexample, Kos17:stability_products, MS17:counterexample_blow_up}.
In general, a relation between Assumpion \ref{assu:bmt} and Castelnuovo theory for projective curves (\cite{Har82:curves_projective_space,CCDG93:genus_projective_curves}) was already observed in \cite{BMT14:stability_threefolds,Tra14:genus_projective_curves_ci}.


\subsubsection*{Strategy of the proof}
\label{subsec:strategy}

The general idea of the proof of Theorem \ref{thm:genus_bound_large_degree_intro} is to study potential walls in tilt stability in the $(\alpha, \beta)$-plane for the ideal sheaf of a curve $C \subset X$, namely codimension one loci at which stability changes. By the Hirzebruch-Riemann-Roch Theorem bounding $\ch_3(\II_C)$ is equivalent to bounding the genus. As mentioned previously, Assumption \ref{assu:bmt} implies that there has to be at least one wall. For each wall there is a semistable subobject $E$ and a semistable quotient $G$. Bounding the third Chern character for $E$ and $G$ induces a bound for $\ch_3(\II_C)$. In order to bound the Chern characters of $E$ and $G$, we study tilt stability for these objects. It turns out that $\Delta(E), \Delta(G) < \Delta(\II_C)$ and since these numbers are non-negative integers, this process has to terminate.

To unify notation among different $X$, we set 
\[
H \cdot \ch(E) := \left(\frac{H^3 \cdot \ch_0(E)}{H^3}, \frac{H^2 \cdot \ch_1(E)}{H^3}, \frac{H \cdot \ch_2(E)}{H^3} ,\frac{\ch_3(E)}{H^3}\right).
\]

The condition $d > k(k-1)$ implies that this process only requires the study of three types of objects which are handled in the next three statements.

\begin{prop}[See Proposition \ref{prop:maximum_ch3_rank1}]
\label{prop:maximum_ch3_rank1_intro}
Let $E \in \Coh^{\beta}(X)$ be a $\nu_{\alpha, \beta}$-semistable object for some $(\alpha, \beta)$ with either $H \cdot \ch(E) = (1,0,-d,e)$ or $H \cdot \ch(E) = (-1,0,d,e)$. Then
\[
e \leq \frac{d(d+1)}{2} - \varepsilon(d,1) = E(d, 1).
\]
\end{prop}

If $E$ is an ideal sheaf of a curve, then Proposition \ref{prop:maximum_ch3_rank1_intro} is the $k = 1$ version of Theorem \ref{thm:genus_bound_large_degree_intro}. Using derived duals (see Proposition \ref{prop:tilt_derived_dual} for details) it is only necessary to prove the case of positive rank.

\begin{thm}[See Theorem \ref{thm:rank_zero_bound}]
\label{thm:rank_zero_bound_intro}
Let $E \in \Coh^{\beta}(X)$ be a $\nu_{\alpha, \beta}$-semistable object for some $(\alpha, \beta)$ with $H \cdot \ch(E) = (0,c,d,e)$, where $c > 0$. Then 
\begin{equation*}
e \leq \frac{c^3}{24} + \frac{d^2}{2c} - \varepsilon\left(d + \frac{c^2}{2}, c\right).
\end{equation*}
\end{thm}

The case $c = 1$ for Theorem \ref{thm:rank_zero_bound_intro} was proved for $\P^3$ in \cite[Lemma 5.4]{Sch15:stability_threefolds}.

\begin{thm}[See Theorem \ref{thm:rank_2_first_cases}]
\label{thm:rank_2_first_cases_intro}
Let $E \in \Coh^{\beta}(X)$ be a $\nu_{\alpha, \beta}$-semistable object for some $(\alpha, \beta)$ with $H \cdot \ch(E) = (2,c,d,e)$. 
\begin{enumerate}[(i)]
\item If $c = -1$, then $d \leq 0$ and 
\[
e \leq \frac{d^2}{2} - d + \frac{5}{24} - \varepsilon\left(d - \frac{1}{2},1\right).
\]
\item If $c = 0$, then $d \leq 0$.
\begin{enumerate}[(a)]
\item If $d = 0$, then $e \leq 0$.
\item If $d = -\tfrac{1}{2}$, then $e \leq \tfrac{1}{6}$.
\item If $d \leq -1$, then
\[
e \leq \frac{d^2}{2} + \frac{5}{24} - \varepsilon\left(d + \frac{1}{2},1\right).
\]
\end{enumerate}
\end{enumerate}
\end{thm}

If $X = \P^3$ and $c = -1$, Theorem \ref{thm:rank_2_first_cases_intro} implies the corresponding case of Theorem \ref{thm:hartshorne_reflexive_bounds} by Hartshorne and Hirschowitz even without the reflexiveness hypothesis. The case $c = 0$ gives a weaker bound here. For just $\P^3$ we could get the stronger bound by a more careful analysis, but it turns out to be wrong for more general threefolds.

All of these statements, including Theorem \ref{thm:genus_bound_large_degree_intro}, are proved with the following strategy. Let $E$ be the object for which we want to bound $\ch_3(E)$. We start by proving the statement for small values of $\Delta(E)$ using $Q_{\alpha, \beta}(E) \geq 0$ whenever $E$ is $\nu_{\alpha, \beta}$-semistable. For larger values of $\Delta(E)$ this strategy provides non-optimal bounds (see \cite{BMT14:stability_threefolds, Mac14:conjecture_p3, Sun16:genus_bounds, Sun16:genus_boundsII}). Instead we study wall-crossing via the following steps. Assume that $\ch_3(E)$ is larger than expected. As explained before, we can assume that $\ch_0(E) \geq 0$.

\begin{enumerate}[(i)]
\item Besides implying the existence of a destabilizing wall, the inequality $Q_{\alpha, \beta}(E) \geq 0$ gives a bound on the rank of the destabilizing subobject (see Lemma \ref{lem:higherRankBound} for details). For example, for ideal sheaves of curves satisfying the assumptions of Theorem \ref{thm:genus_bound_large_degree_intro} the subobject can only be of rank one or two. Let
\[
0 \to F \to E \to G \to 0
\]
be the destabilizing sequence. The argument is always symmetric in $F$ and $G$ and without loss of generality, we can assume $\ch_0(F) \geq 1$.
\item Using the fact that $\ch_1^{\beta}(F) > 0$ and $\ch_1^{\beta}(G) > 0$ for any $\beta$ along the wall, we obtain a lower and upper bound on $\ch_1(F)$.
\item The Bogomolov inequality $\Delta(F) \geq 0$ yields an upper bound on $\ch_2(F)$. The inequality $\Delta(G) \geq 0$ yields another bound on $\ch_2(F)$ (lower or upper bound depending on the rank of $G$). Moreover, the fact that a wall cannot lie in the area $Q_{\alpha, \beta}(E) < 0$ leads to a lower bound on $\ch_2(F)$. Overall, this reduces the problem to finitely many walls.
\item Next, we use some previously obtained bounds for $\ch_3(F)$ and $\ch_3(G)$ to bound $\ch_3(E)$.
\item In general, the walls are linearly ordered. The last step is to check that the previous bound is decreasing with this ordering, and the largest wall still provides a contradiction.
\end{enumerate}

We prove the statements in the following order. First, the case $c = 1$ in Theorem \ref{thm:rank_zero_bound_intro} is proved via $Q_{\alpha, \beta}(E) \geq 0$ (see Lemma \ref{lem:bound_torsion_on_plane} for details). Next, we prove Proposition \ref{prop:maximum_ch3_rank1_intro}. It turns out that the subobjects are also of rank one, for which we use induction, and the  bounds on the quotients follow from the $c = 1$ case in Theorem \ref{thm:rank_zero_bound_intro}. After that, we use Proposition \ref{prop:maximum_ch3_rank1_intro} on both subobjects and quotients, to prove Theorem \ref{thm:rank_zero_bound_intro}. All of the previous statements are used to prove Theorem \ref{thm:rank_2_first_cases_intro} with an induction on $\Delta(E)$. Finally, Theorem \ref{thm:genus_bound_large_degree_intro} can be proved using the same steps again.


\subsubsection*{Hartshorne--Hirschowitz Conjecture}
\label{subsec:HHConjecture}

Coming back to the case of projective space, our aim is to improve our techniques towards a possible approach to the Hartshorne--Hirschowitz Conjecture, namely to the case in which $d\leq k(k-1)$. Let us first recall the statement of the conjecture (see \cite{Har87:space_curvesII,HH88:genus_bound,Har88:stable_reflexive_3}).

For given integers $d$ and $k$, let $G(d,k)$ be the maximal genus of an integral curve $C \subset \P^3$ with degree $d$ such that $C$ is not contained in a surface of degree smaller than $k$. It is easy to check that $d \geq \frac{1}{6}(k^2 + 4k + 6)$.
For $d > k(k-1)$, $G(d,k)$ is given by Theorem \ref{thm:GP}, since a curve with genus $G(d,k)$ always exists under those assumptions.
If
\[
\frac{1}{6}(k^2 + 4k + 6) \leq d < \frac{1}{3}(k^2 + 4k + 6),
\]
then it is not hard to find a bound from above for $G(d,k)$, but currently it is still not known in full generality if this bound is sharp (see \cite{Har87:space_curvesII,BBEMR97:maximum_genus_rangeA} for results in this direction). We are interested in the remaining case
\begin{equation}\label{eq:RangeB}
\frac{1}{3}(k^2 + 4k + 6) \leq d \leq k(k-1).
\end{equation}
Note that this case only makes sense for $k \geq 5$. We first introduce another error term as follows. For any integer $c \in \Z$, let
\[
\delta (c) :=  \begin{cases}
3 &\text{if $c=1,3$}\\
1 &\text{if $c \equiv 2 \ (\mod 3)$} \\
0 &\text{otherwise}.
\end{cases}
\]
Then, for any integers $k\geq 5$ and $f \in [k-1, 2k - 5]$, we define integers
\begin{align*}
A(k,f) &:= \frac{1}{3}(k^2 - kf + f^2 - 2k + 7f + 12 + \delta(2k-f-6)), \\
B(k,f) &:= \frac{1}{3}(k^2 - kf + f^2  + 6f + 11 + \delta(2k-f-7)).
\end{align*}

A straightforward computation shows that $A(k,f)$ is an increasing function for $f \in [k-1, 2k - 5]$ and that it partitions our range of $d$ in \eqref{eq:RangeB}: $A(k,k-1) = \lceil\frac{1}{3}(k^2 + 4k + 6)\rceil$ and $A(k,2k-5)=k(k-1)+1$. Moreover, we have $A(k,f) < B(k,f) \leq A(k, f+1)$.

\begin{conj}[Hartshorne--Hirschowitz, see Conjecture \ref{conj:hartshorne_rangeB}]
\label{conj:hartshorne_rangeB_intro}
Let $d,k>0$ be integers.
Assume that $A(k,f) \leq d < A(k,f+1)$ for some $f \in [k-1,2k-6]$. Then
\[
G(d,k) = d(k-1) + 1 - \binom{k+2}{3} + \binom{f-k+4}{3} + h(d),
\]
where
\[
h(d) = \begin{cases}
0 & \text{if } A(k,f) \leq d \leq B(k,f) \\
\frac{1}{2}(d-B(k,f))(d-B(k,f) + 1) & \text{if } B(k,f) \leq d < A(k,f+1)).
\end{cases}
\]
\end{conj}

By \cite{HH88:genus_bound} it is known that there exist such curves with genus $G(d,k)$ as in the conjecture. Therefore, one only has to prove that every curve satisfies this bound. This is known for a few values of $f$: the cases $f=k-1,k$ were proved in \cite{Har88:stable_reflexive_3}, while the case $f=2k-6$ is in \cite{GP82:postulation_courbes_gauches}, $f=2k-7$ in \cite{Ell91:genre_courbes_gauches}, $f=2k-8,2k-9$ in \cite{ES92:sections_planes_genre}, and $f=2k-10$ in \cite{Str93:plane_sections}.

This conjecture is partially based on the fact that this bound is obtained for curves with an extension
\[
0 \to \OO_{\P^3}(-f-4) \to E \to \II_C \to 0
\]
by bounding the third Chern character of the reflexive sheaf $E$. This sequence constitutes a potential wall in tilt stability for $\II_C$, because in our abelian category this corresponds to an exact sequence
\[
0 \to E \to \II_C \to \OO_{\P^3}(-f-4)[1] \to 0.
\]
It turns out that our approach requires to study walls above or below this wall with slightly different methods, and therefore, we suggest the following two questions. We need one extra bit of notation (see Theorem \ref{thm:Bertram} for a more in detail description of walls and their possible shapes). Given two elements $E, F \in \Db(\P^3)$, let $W(E,F)$ be the locus in the $(\alpha,\beta)$-plane where $E$ and $F$ have the same $\nu_{\alpha,\beta}$-slope. In the cases we will be interested in these loci are semicircles with center on the $\beta$-axis.

\begin{question}
\label{ques:largeWallRangeB}
Assume the hypothesis of Conjecture \ref{conj:hartshorne_rangeB_intro}. Let $C$ be an integral curve of genus $g$ and degree $d$ such that $H^0(\II_C(k - 1)) = 0$. If $\II_C$ is destabilized in tilt stability above or at the numerical wall $W(\II_C, \OO(-f-4)[1])$, does $g \leq G(d,k)$ hold?
\end{question}

Our second main result is an affirmative answer to this question in a smaller range.

\begin{thm}[See Theorem \ref{thm:largeWallRangeB}]
\label{thm:largeWallRangeB_intro}
Question \ref{ques:largeWallRangeB} has an affirmative answer if $A(k, f) \leq d \leq B(k,f)$, and the base field has characteristic $0$.
\end{thm}

A full proof of the conjecture also requires to study walls below $W(\II_C, \OO(-f-4)[1])$. We suggest the following approach.

\begin{question}
\label{ques:decreasing}
Assume the hypothesis of Conjecture \ref{conj:hartshorne_rangeB_intro}, and let $C$ be destabilized below the wall $W(\II_C, \OO(-f-4)[1])$. Is the maximal possible genus of $C$ decreasing with the size of the wall?
\end{question}

All arguments in Section \ref{sec:classical_bounds} suggest that the maximum $\ch_3$ for semistable objects is decreasing with the size of the wall even beyond ideal sheaves. The most serious obstacle for studying this question is the fact that in general destabilizing subobjects can be reflexive sheaves of high rank. Beyond rank two results are scarce. Another problem is that we would need to consider more general bounds for not necessarily integral curves, but in our setting this is probably more approachable. In any case, a positive answer to both Question \ref{ques:largeWallRangeB} and Question \ref{ques:decreasing} would indeed prove Conjecture \ref{conj:hartshorne_rangeB_intro}, since if $C$ does not satisfy the conjecture, then $I_C$ will be destabilized at a certain point.

In order to handle rank two objects in the proof of Theorem \ref{thm:largeWallRangeB_intro} we need the following result (see \cite[Theorem 0.1]{Har82:stable_reflexive_2}, \cite{GP82:postulation_courbes_gauches}, \cite[Theorem 3.2, 3.3]{Har87:space_curvesII}, \cite{HH88:genus_bound}, and \cite[Theorem 1.1]{Har88:stable_reflexive_3}).

\begin{thm}
\label{thm:hartshorne_reflexive_bounds}
Assume that the base field has characteristic $0$.
Let $E \in \Coh(\P^3)$ be a rank two reflexive sheaf with $\ch(E) = (2,c,d,e)$, $c \geq -1$, and $H^0(E) = 0$. Then $d \leq \tfrac{1}{6} c^2 - \tfrac{2}{3} c - 1 - \frac{\delta(c)}{3}$. Moreover, 
\begin{enumerate}
\item if $\tfrac{1}{6}c^2 - c - \tfrac{8}{3} - \tfrac{\delta(c-1)}{3} \leq d \leq \tfrac{1}{6} c^2 - \tfrac{2}{3}c - 1 - \tfrac{\delta(c)}{3}$, then $h^2(E) = 0$ and
\[
e \leq -\frac{11c}{6} - 2d - 2,
\]
\item if $d \leq \tfrac{1}{6}c^2 - c - \tfrac{8}{3} - \tfrac{\delta(c-1)}{3}$, then
\[
h^2(E) \leq \frac{(c^2 - 6c - 6d - 2\delta(c-1) - 10)(c^2 - 6c - 6d - 2\delta(c-1) - 16)}{72}
\]
and
\[
e \leq \frac{c^4}{72} - \frac{c^3}{6} + \frac{5c^2}{36} + \frac{c}{3} - \frac{c^2d}{6} + cd + \frac{d^2}{2} + \frac{d}{6} + \frac{2}{9} - \frac{\delta(c-1)}{18}(c^2 - 6c - 6d - \delta(c-1) - 13).
\]
\end{enumerate}
Furthermore, these bounds are strict in the sense that there are rank two stable reflexive sheaves $E$ with $H^0(E) = 0$ reaching them in all cases.
\end{thm}

A more detailed tilt stability version of Theorem \ref{thm:hartshorne_reflexive_bounds} is surely necessary to answer Questions \ref{ques:largeWallRangeB} and \ref{ques:decreasing} in general.
 
Finally, we illustrate our approach in one example. We prove Conjecture \ref{conj:hartshorne_rangeB_intro} in the case $d = A(k, 2k - 11)$ when $k \geq 31$ in Proposition \ref{prop:special_case_2k10}. For a fixed $k$ this is the largest degree for which the conjecture is unknown.  We have no doubt that a slightly more careful analysis would also handle the cases $k < 31$.


\subsubsection*{Acknowledgements}
We would like to thank Roman Bezrukavnikov for originally proposing this question during a talk by the first author, and Arend Bayer, Izzet Coskun, Patricio Gallardo, Mart\'i Lahoz, Luigi Lombardi, C\'esar Lozano Huerta, John Ottem, Giuseppe Pareschi, Paolo Stellari, and Xiaolei Zhao for very useful discussions. We also thank the referee for useful suggestions. The authors were partially supported by the NSF grant DMS-1523496 and by the NSF FRG-grant DMS-1664215. Parts of the paper were written while the first author was holding a Poincar\'e Chair from the Institut Henri Poincar\'e and the Clay Mathematics Institute.
The authors would also like to acknowledge the following institutions: Institut Henri Poincar\'e, Northeastern University, and University of Texas.


\subsubsection*{Notation}

\begin{center}
  \begin{tabular}{ r l }
    $X$ & smooth projective threefold over an algebraically closed field $\F$ \\
    $H$ & fixed ample divisor on $X$ \\
    $\Db(X)$ & bounded derived category of coherent sheaves on $X$ \\
    $\HH^{i}(E)$ & the $i$-th cohomology group of a complex $E \in \Db(X)$ \\
    $H^i(E)$ & the $i$-th sheaf cohomology group of a complex $E \in \Db(X)$ \\
    $\D(\cdot)$ & the derived dual $\RlHom(\cdot, \OO_X)[1]$ \\
    $\ch(E)$ & Chern character of an object $E \in \Db(X)$  \\
    $\ch_{\leq l}(E)$ & $(\ch_0(E), \ldots, \ch_l(E))$ \\
    $H \cdot \ch(E)$ & $\left(\frac{H^3 \cdot \ch_0(E)}{H^3}, \frac{H^2 \cdot \ch_1(E)}{H^3}, \frac{H \cdot \ch_2(E)}{H^3}, \frac{\ch_3(E)}{H^3}\right)$ \\
    $H \cdot \ch_{\leq l}(E)$ & $\left(\frac{H^3 \cdot \ch_0(E)}{H^3}, \ldots, \frac{H^{3-l} \cdot \ch_l(E)}{H^3}\right)$
  \end{tabular}
\end{center}

\section{Background on stability conditions}
\label{sec:background}

In \cite{BMT14:stability_threefolds} the notion of tilt stability was introduced as an auxiliary notion in between slope stability and a conjectural construction of Bridgeland stability on threefolds. It turns out to be useful in its own right as pointed out, for example, in \cite{Sch15:stability_threefolds,Xia16:twisted_cubics}. In this section, we give a quick introduction of tilt stability and its basic properties. We will restrict to the case of Picard rank one, even though the theory can be developed more generally.

\subsection{Definition}
\label{subsec:definition}

Let $X$ be a smooth projective threefold over an algebraically closed field $\F$. The first assumption we will make in this article is to restrict its possible divisors and curves.

\begingroup
\def\thetheorem{\ref{assu:div_curves}}
\begin{assuInt}
The N\'eron-Severi group is generated by the class of an ample divisor $H$, i.e., $N^1(X) = \Z \cdot H$. Moreover, the Chern character of any sheaf $E \in \Coh(X)$ satisfies $\ch_2(E) \in \tfrac{1}{2} H^2 \cdot \Z$ and $\ch_3(E) \in \tfrac{1}{6} H^3 \cdot \Z$.
\end{assuInt}
\endgroup


This assumption is not needed for the results in this preliminary section, but it will be important for the remainder of the article. It holds in particular for $\P^3$, for principally polarized abelian threefolds of Picard rank one, and for Fano threefolds of Picard rank one, index two, and degree one or two.

The classical \emph{slope} for a coherent sheaf $E \in \Coh(X)$ is defined as
\[
\mu(E) := \frac{H^2\cdot \ch_1(E)}{H^3 \cdot \ch_0(E)},
\]
where division by zero is interpreted as $+\infty$. As usual a coherent sheaf $E$ is called \emph{slope (semi)stable} if for any non trivial proper subsheaf $F \subset E$ the inequality $\mu(F) < (\leq) \mu(E/F)$ holds. We will assume that the following assertion holds. In characteristic zero this is nothing but a consequence of the classical Bogomolov inequality (\cite{Rei78:bogomolov,Bog78:inequality,Gie79:bogomolov}).
In positive characteristic it holds, for example, in  $\P^3$ and abelian threefolds (\cite{MR82:homogeneous_bundles_char_p,Lang04:positive_char})\footnote{In \cite{Lang04:positive_char}, a general Bogomolov inequality is proved over any field, by adding an extra term to the inequality in Assumption \ref{assu:classic_bg}. However, this is generally not enough to define tilt-stability.}.

\begingroup
\def\thetheorem{\ref{assu:classic_bg}}
\begin{assuInt}
Any slope semistable sheaf $E \in \Coh(X)$ satisfies
\[
\Delta(E) := \frac{(H^2 \cdot \ch_1(E))^2 - 2(H^3 \cdot \ch_0(E))(H \cdot \ch_2(E))}{(H^3)^2} \geq 0.
\]
\end{assuInt}
\endgroup

Note that, by using Assumption \ref{assu:div_curves}, $\Delta(E) \in \Z$. Let $\beta$ be an arbitrary real number. Then the twisted Chern character $\ch^{\beta}$ is defined to be $e^{-\beta H} \cdot \ch$. Explicitly:
\begin{align*}
\ch^{\beta}_0 &= \ch_0, \ \ch^{\beta}_1 = \ch_1 - \beta H \cdot \ch_0, \ \ch^{\beta}_2 = \ch_2 - \beta H \cdot \ch_1 + \frac{\beta^2}{2} H^2 \cdot \ch_0,\\
\ch^{\beta}_3 &= \ch_3 - \beta H \cdot \ch_2 + \frac{\beta^2}{2} H^2 \cdot \ch_1 - \frac{\beta^3}{6} H^3 \cdot \ch_0.
\end{align*}

The process of tilting is used to construct a new heart of a bounded t-structure. For more information on the general theory of tilting we refer to \cite{HRS96:tilting,BvdB03:functors}.
A torsion pair is defined by
\begin{align*}
\TT_{\beta} &:= \{E \in \Coh(X) : \text{any quotient $E \onto G$ satisfies $\mu(G) > \beta$} \}, \\
\FF_{\beta} &:=  \{E \in \Coh(X) : \text{any subsheaf $F \subset E$ satisfies $\mu(F) \leq \beta$} \}.
\end{align*}
The heart of a bounded t-structure is given as the extension closure $\Coh^{\beta}(X) := \langle \FF_{\beta}[1],\TT_{\beta} \rangle$. Let $\alpha > 0$ be a positive real number. The tilt slope is defined as
\[
\nu_{\alpha, \beta} := \frac{H \cdot \ch^{\beta}_2 - \frac{\alpha^2}{2} H^3
\cdot \ch^{\beta}_0}{H^2 \cdot \ch^{\beta}_1}.
\]
Similarly as before, an object $E \in \Coh^{\beta}(X)$ is called \emph{tilt-(semi)stable} (or \emph{$\nu_{\alpha,\beta}$-(semi)stable}) if for any non trivial proper subobject $F \subset E$ the inequality $\nu_{\alpha, \beta}(F) < (\leq) \nu_{\alpha, \beta}(E/F)$ holds.
Assumption \ref{assu:classic_bg} implies that this notion of stability is well-defined and that it shares many properties with slope-stability for sheaves (\cite{BMT14:stability_threefolds}); in particular, Harder-Narasimhan filtrations exist and stability is open when varying $(\alpha,\beta)$.

\subsection{Walls and inequalities}
\label{subsec:walls}

A version of the classical Bogomolov inequality also holds in tilt stability assuming that it holds for slope semistable sheaves.

\begin{thm}[{Bogomolov inequality for tilt stability, \cite[Corollary 7.3.2]{BMT14:stability_threefolds}}]
Assume that Assumption \ref{assu:classic_bg} holds.
Then, any $\nu_{\alpha, \beta}$-semistable object $E \in \Coh^{\beta}(X)$ satisfies $\Delta(E) \geq 0$.
\end{thm}

The following inequality involving the third Chern character was conjectured in \cite{BMT14:stability_threefolds} and was brought into the following form in \cite{BMS16:abelian_threefolds}. 

\begingroup
\def\thetheorem{\ref{assu:bmt}}
\begin{assuInt}
For any $\nu_{\alpha,\beta}$-semistable object $E \in \Coh^{\beta}(X)$ the inequality
\[
Q_{\alpha, \beta}(E) := \alpha^2 \Delta(E) + \frac{4(H \cdot \ch_2^{\beta}(E))^2}{(H^3)^2} - \frac{6 (H^2\cdot \ch_1^{\beta}(E)) \ch_3^{\beta}(E)}{(H^3)^2} \geq 0
\]
holds.
\end{assuInt}
\endgroup

In our setting of Picard rank one this is known to hold in characteristic zero\footnote{Some of the arguments in the known proofs do generalize directly to positive characteristic. For example, in $\P^3$ Assumption \ref{assu:bmt} holds over any field.} for both Fano threefolds (\cite{Mac14:conjecture_p3,Sch14:conjecture_quadric,Li15:conjecture_fano_threefold}) and abelian threefolds (\cite{MP15:conjecture_abelian_threefoldsI,MP16:conjecture_abelian_threefoldsII,BMS16:abelian_threefolds}).

Let $\Lambda \subset \Z \oplus \Z \oplus \tfrac{1}{2} \Z$ be the image of the map $H \cdot \ch_{\leq 2}$. Notice that $\nu_{\alpha, \beta}$ factors through $H \cdot \ch_{\leq 2}$. Varying $(\alpha, \beta)$ changes the set of stable objects. A \textit{numerical wall} in tilt stability with respect to a class $v \in \Lambda$ is a non trivial proper subset $W$ of the upper half plane given by an equation of the form $\nu_{\alpha, \beta}(v) = \nu_{\alpha, \beta}(w)$ for another class $w \in \Lambda$. We will usually write $W = W(v, w)$.

A subset $S$ of a numerical wall $W$ is called an \textit{actual wall} if the set of semistable objects with class $v$ changes at $S$. The structure of walls in tilt stability is rather simple. Part (i) - (iv) is usually called Bertram's Nested Wall Theorem and appeared in \cite{Mac14:nested_wall_theorem}, while part (v) and (vi) can be found in \cite[Appendix A]{BMS16:abelian_threefolds}.

\begin{thm}[Structure Theorem for Walls in Tilt Stability]
\label{thm:Bertram}
Let $v \in \Lambda$ be a fixed class. All numerical walls in the following statements are with respect to $v$.
\begin{enumerate}[(i)]
  \item Numerical walls in tilt stability are either semicircles with center on the $\beta$-axis or rays parallel to the $\alpha$-axis. Moreover, a semicircular wall with radius $\rho$ and center $s$ satisfies
\[
\frac{(H^3 \cdot \ch_0(v))^2}{(H^3)^2} \rho^2 + \Delta(v) = \frac{(H^3 \cdot \ch_0(v) s - H^2 \cdot \ch_1(v))^2}{(H^3)^2}.
\]
 If $v_0 \neq 0$, there is exactly one numerical vertical wall given by $\beta = v_1/v_0$. If $v_0 = 0$, there is no actual vertical wall.
   \item The curve $\nu_{\alpha, \beta}(v) = 0$ is given by a hyperbola, which may be degenerate. Moreover, this hyperbola intersects all semicircular walls at their top point.
  \item If two numerical walls given by classes $w,u \in \Lambda$ intersect, then $v$, $w$ and $u$ are linearly dependent. In particular, the two walls are completely identical.
  \item If a numerical wall has a single point at which it is an actual wall, then all of it is an actual wall.
  \item If there is an actual wall numerically defined by an exact sequence of tilt semistable objects $0 \to F \to E \to G \to 0$ such that $H \cdot \ch_{\leq 2}(E) = v$, then 
  \[\Delta(F) + \Delta(G) \leq \Delta(E).\]
  Moreover, equality holds if and only if $H \cdot \ch_{\leq 2}(G) = 0$.
  \item If $\Delta(E) = 0$, then $E$ can only be destabilized at the unique numerical vertical wall. In particular, line bundles, respectively their shifts by one, are tilt semistable everywhere.
\end{enumerate}
\end{thm}

If $W = W(v,w)$ is a semicircular wall in tilt stability for two numerical classes $v, w \in \Lambda$, then we denote its \emph{radius} by $\rho_W = \rho(v,w)$ and its \emph{center} on the $\beta$-axis by $s_W = s(v,w)$. The structure of the locus $Q_{\alpha, \beta}(E) = 0$ fits right into this picture; indeed, a straightforward computation shows the following.

\begin{lem}
Let $E \in \Db(X)$. The equation $Q_{\alpha, \beta}(E) = 0$ is equivalent to
\[
\nu_{\alpha, \beta}(E) = \nu_{\alpha, \beta}\left(\frac{H^2 \cdot \ch_1(E)}{H^3}, \frac{2H \cdot \ch_2(E)}{H^3}, \frac{3\ch_3(E)}{H^3}\right).
\]
In particular, $Q_{\alpha, \beta}(E) = 0$ describes a numerical wall in tilt stability.
\end{lem}

\subsection{Further properties}
\label{subsec:further_properties}

We will need the following modification of \cite[Proposition 8.3]{CH16:ample_cone_plane}. It is a highly convenient tool to control the rank of destabilizing subobjects.

\begin{lem}
\label{lem:higherRankBound}
Assume that a tilt semistable object $E$ is destabilized by either a subobject $F \into E$ or a quotient $E \onto F$ in $\Coh^{\beta}(X)$ inducing a non empty semicircular wall $W$. Assume further that $\ch_0(F) > \ch_0(E) \geq 0$. Then the inequality
\[
\rho_W^2 \leq \frac{\Delta(E)}{4 \ch_0(F) (\ch_0(F) - \ch_0(E))}
\]
holds.
\end{lem}

\begin{proof}
For all $(\alpha, \beta) \in W$ we have the inequalities $H^2 \cdot \ch_1^{\beta}(E) \geq H^2 \cdot \ch_1^{\beta}(F) \geq 0$. This can be rewritten as
\[
H^2 \cdot \ch_1(E) + \beta (H^3 \cdot \ch_0(F) - H^3 \cdot \ch_0(E)) \geq H^2 \cdot \ch_1(F) \geq \beta H^3 \cdot \ch_0(F).
\]
Since $H^2 \cdot \ch_1(F)$ is independent of $\beta$ we can maximize the right hand side and minimize the left hand side individually in the full range of $\beta$ between $s_W - \rho_W$ and $s_W + \rho_W$. By our assumptions this leads to
\[
H^2 \cdot \ch_1(E) + (s_W - \rho_W) (H^3 \cdot \ch_0(F) - H^3 \cdot \ch_0(E)) \geq (s_W + \rho_W) H^3 \cdot \ch_0(F).
\]
By rearranging the terms and squaring we get
\[
(2H^3 \cdot \ch_0(F) - H^3 \cdot \ch_0(E))^2 \rho_W^2 \leq (H^2 \cdot \ch_1(E) - H^3 \cdot \ch_0(E) s_W)^2 = (H^3 \cdot \ch_0(E))^2 \rho_W^2 + (H^3)^2 \Delta(E).
\]
The claim follows by simply solving for $\rho_W^2$.
\end{proof}

Objects that are stable for $\alpha \gg 0$ are closely related to slope semistable objects.

\begin{lem}[{\cite[Lemma 2.7]{BMS16:abelian_threefolds}}]
\label{lem:large_volume_limit}
If $E \in \Coh^{\beta}(X)$ is $\nu_{\alpha, \beta}$-semistable
for all $\alpha \gg 0$, then it satisfies one of the following conditions:
\begin{enumerate}[(i)]
\item $\HH^{-1}(E)=0$ and $\HH^0(E)$ is a torsion-free slope semistable sheaf,
\item $\HH^{-1}(E)=0$ and $\HH^0(E)$ is a torsion sheaf, or
\item $\HH^{-1}(E)$ is a torsion-free slope semistable sheaf, and $\HH^0(E)$ is either $0$ or a torsion sheaf supported in dimension less than or equal to one.
\end{enumerate}
Conversely, assume that $E \in \Coh(X)$ is a torsion-free slope stable sheaf and $\beta < \mu(E)$. Then $E \in \Coh^{\beta}(X)$ is $\nu_{\alpha,\beta}$-stable for $\alpha \gg 0$.
\end{lem}

Instead of directly using the usual derived dual, we define
\begin{align*}
\D: \Db(\P^3) &\to \Db(\P^3), \\
E &\mapsto \RlHom(E, \OO)[1].
\end{align*}

\begin{prop}[{\cite[Proposition 5.1.3]{BMT14:stability_threefolds}}]
\label{prop:tilt_derived_dual}
Assume $E \in \Coh^{\beta}(X)$ is $\nu_{\alpha, \beta}$-semistable with $\nu_{\alpha, \beta}(E) \neq \infty$. Then there is a $\nu_{\alpha, -\beta}$-semistable object $\tilde{E} \in \Coh^{-\beta}(X)$ and a sheaf $T$ supported in dimension $0$ together with a triangle
\[
\tilde{E} \to \D(E) \to T[-1] \to \tilde{E}[1].
\]
\end{prop}




Finally, the following elementary result will be used several times.

\begin{lem}
\label{lem:stable_minimal_ch1}
Let $\beta_0 = \frac{p}{q} \in \mathbb{Q}$, with $p,q$ coprime. Let $E \in \Coh^{\beta_0}(X)$ be such that either $\ch_1^{\beta_0}(E) = \frac{1}{q} \cdot H$ or $\ch_1^{\beta_0}(E) = 0$. Then, $E$ does not have any wall on the ray $\beta = \beta_0$ unless it is the unique vertical wall. More precisely, for all $\alpha_1, \alpha_2 > 0$, $E$ is $\nu_{\alpha_1, \beta_0}$-(semi)stable if and only if it is $\nu_{\alpha_2, \beta_0}$-(semi)stable.
\end{lem}

\section{Classical bounds beyond projective space}
\label{sec:classical_bounds}

The main goal of this section is to prove Theorem \ref{thm:genus_bound_large_degree} below.
Let $X$ be a smooth projective variety over an algebraically closed field $\F$ for which Assumptions \ref{assu:div_curves}, \ref{assu:classic_bg}, and \ref{assu:bmt} hold.
The examples to keep in mind is $X$ being either $\P^3$ or a principally polarized abelian threefold in characteristic zero.
We denote by $H$ the ample generator of $\mathrm{NS}(X)$, e.g., the hyperplane class for $\P^3$ or a $\Theta$-divisor for an abelian threefold.

The \emph{degree} of a hypersurface $Y \subset X$ is defined as $k = (H^2 \cdot Y)/H^3$. If $C \subset X$ is a one-dimensional closed subscheme, we define its \emph{degree} as $d = (H \cdot C)/H^3$. Note that $k \in \Z$, and $d \in \tfrac{1}{2} \Z$. If $X = \P^3$, we even have $d \in \Z$. The \emph{arithmetic genus} of $C$ is defined by $g = 1 - \chi(\OO_C)$. By the Hirzebruch-Riemann-Roch Theorem, we know
\[
g = 1 - \chi(\OO_C) = 1 + \frac{K_X \cdot C}{2} - \ch_3(\OO_C) = 1 + \frac{K_X \cdot C}{2} + \ch_3(\II_C).
\]
Therefore, bounding $g$ is equivalent to bounding $\ch_3(\II_C)$. The following error terms for any $d \in \tfrac{1}{2} \Z$ will occur in this and subsequent statements:
\[
\varepsilon(d, 1) = \begin{cases}
\frac{1}{24} &\text{if $d \notin \Z$,} \\
0 &\text{if $d \in \Z$,}
\end{cases}
\]
and for $k \geq 1$
\begin{align*}
\tilde{\varepsilon}(d,k) &= \frac{1}{2}f\left(k-f-1+\frac{f}{k}\right), \\
\varepsilon(d, k) &= \tilde{\varepsilon}(d,k) + \varepsilon(d, 1),
\end{align*}
where $d \equiv -f (\mod k)$ and $0 \leq f < k$. The inclusion of $\varepsilon(d,1)$ is related to the fact that Assumption \ref{assu:div_curves} says $\ch_3(E) \in \tfrac{1}{6}H^3$ for any object $E \in \Db(X)$. It simply constitutes a rounding term, all statements can be equivalently stated with $\tilde{\varepsilon}(d,k)$, and we will do so throughout the proofs for simplification.

\begin{thm}
\label{thm:genus_bound_large_degree}
Let $k \in \Z_{>0}$ and $d \in \frac{1}{2}\Z_{>0}$, and let $C \subset X$ be an integral curve of degree $d$. Further, assume
\begin{itemize}
\item $H^0(X, I_C((k-1)H')=0$ for any divisor $H'$ in the same numerical class as $H$, and
\item $d > k(k-1)$.
\end{itemize}
Then
\[
\frac{\ch_3(\II_C)}{H^3} \leq E(d,k) := \frac{d^2}{2k} + \frac{dk}{2} - \varepsilon(d,k).
\]
\end{thm}

We will also denote 
\[
\tilde{E}(d, k) = \frac{d^2}{2k} + \frac{dk}{2} - \tilde{\varepsilon}(d,k).
\]

\subsection{Stable rank one objects}\label{subsec:rk1}

Recall that for any curve $C \subset \P^3$ of degree $d$ and genus $g$ the inequality
\[
g \leq \frac{(d-1)(d-2)}{2}
\]
holds. We will prove the following generalization to tilt semistable objects on $X$ whose Chern character is that of an ideal sheaf. This corresponds to the $k = 1$ version of Theorem \ref{thm:genus_bound_large_degree}.

\begin{prop}
\label{prop:maximum_ch3_rank1}
Let $E \in \Coh^{\beta}(X)$ be a $\nu_{\alpha, \beta}$-semistable object for some $(\alpha, \beta)$ with either $H \cdot \ch(E) = (1,0,-d,e)$ or $H \cdot \ch(E) = (-1,0,d,e)$. Then
\[
e \leq \frac{d(d+1)}{2} - \varepsilon(d,1) = E(d, 1).
\]
\end{prop}

Note that the condition $\ch_1(E) = 0$ is no real restriction, since it can always be achieved by tensoring with an appropriate line bundle. In order to prove the proposition, we will first need to deal with torsion sheaves supported on a hypersurface of class $H$. The following lemma generalizes \cite[Lemma 5.4]{Sch15:stability_threefolds} from $\P^3$ to $X$.

\begin{lem}
\label{lem:bound_torsion_on_plane}
Let $E \in \Coh^{\beta}(X)$ be a $\nu_{\alpha, \beta}$-semistable object with $H \cdot \ch(E) = (0,1,d,e)$. Then 
\[
e \leq \frac{1}{24} + \frac{d^2}{2} - \varepsilon\left(d + \frac{1}{2}, 1\right).
\]
\end{lem}

\begin{proof}
As previously, the term 
\[
\varepsilon\left(d + \frac{1}{2}, 1\right)
\]
is simply a rounding term, and it is enough to show
\[
e \leq \frac{1}{24} + \frac{d^2}{2}.
\] 
For any $(\alpha, \beta)$ in the semidisk
\[
\alpha^2 + (\beta - d)^2 \leq \frac{1}{4},
\]
the inequality $Q_{\alpha, \beta}(E) \geq 0$ implies the claim. We are done if we can show that there is no wall outside this semidisk. Lemma \ref{lem:higherRankBound} implies that if a wall has radius squared larger than $\tfrac{1}{4}\Delta(E) = \tfrac{1}{4}$, it must be induced by a rank zero subobject. But such a subobject destabilizes $E$ either for all $(\alpha, \beta)$ or none.

\end{proof}

\begin{proof}[Proof of Proposition \ref{prop:maximum_ch3_rank1}]
Assume that $\ch_0(E) = -1$. By Proposition \ref{prop:tilt_derived_dual}, there is a sheaf $T$ supported in dimension $0$ and a $\nu_{\alpha, - \beta}$-semistable object $\tilde{E}$ together with a distinguished triangle
\[
T \to \tilde{E} \to \D(E) \to T[1].
\]
We have $\ch(\tilde{E}) = (1, 0, - \ch_2(E), \ch_3(E) + \ch_3(T))$. Thus, it is enough to deal with the $\ch_0(E) = 1$ case. The proof of that case is by induction on $d$. If $d = 0$, then $Q_{\alpha, \beta}(E) \geq 0$ is equivalent to $e \leq 0$. Let $d = \tfrac{1}{2}$. Then we have $\ch^{-1}(E) = (1,1,0,e - \tfrac{1}{3})$. All semicircular walls intersect the vertical line $\beta = -1$. Since there is no vertical wall at $\beta = -1$, this implies that $E$ has to be stable for $\beta = -1$ regardless of $\alpha$. Therefore, we can use $Q_{0,-1}(E) \geq 0$ to obtain
\[
e \leq \frac{1}{3}.
\]
Let $d \geq 1$ and assume for a contradiction that
\[
e > \frac{d(d+1)}{2}.
\]
Note again that the reason that we can ignore $\varepsilon(d,1)$ is the fact that it is just a rounding term. We write
\[
s_Q := -\frac{3e}{2d}.
\]
Assumption \ref{assu:bmt} implies 
\[
\alpha^2 + \left(\beta - s_Q \right)^2 \geq \frac{9e^2 - 8d^3}{4d^2} =: \rho^2_Q.
\]
This equation describes the complement of a semidisk and there are no stable objects inside. Therefore, a potentially stable object must become strictly semistable at some larger semicircular wall. We will show that $E$ can only be destabilized by either a subobject $F \into E$ or quotient $E \onto F$ with $\ch_0(F) = 1$. Let $\ch_0(F) = r \geq 1$ and $H^2 \cdot \ch_1(F) = xH^3$. We have
\[
\rho^2_Q - \frac{d}{4} = \frac{9e^2 - 8d^3}{4d^2} - \frac{d}{4} > \frac{9}{16}(d - 1)^2 \geq 0
\]
and can deduce $r = 1$ from Lemma \ref{lem:higherRankBound}.

The center of the numerical wall between $E$ and $\OO(-2H)$ is given by
\[
s(E, \OO(-2H)) =  -\frac{d}{2} - 1.
\]
This compares to $s_Q$ as follows:
\[
s(E, \OO(-2H)) - s_Q > \frac{d-1}{4} \geq 0.
\]
Therefore, all walls are bigger than the numerical wall with $\OO(-2H)$, i.e., $x = -1$. To simplify notation, we will write
\[
H \cdot \ch(F) = (1, 0, -y, z) \cdot (H \cdot \ch(\OO(-H))) = \left(1, -1, -y + \frac{1}{2}, y + z - \frac{1}{6} \right).
\]
The center of the wall is given by
\[
s(E, F) = y - d - \frac{1}{2}.
\]
In order to be outside the semidisk with no semistable objects, the inequality
\[
y - d - \frac{1}{2} \leq -\frac{3e}{2d} < -\frac{3}{4}(d+1)
\]
needs to hold. This implies
\[
0 \leq y < \frac{d-1}{4} \leq \frac{d-1}{2}.
\]
By induction we know $z \leq \tfrac{y(y+1)}{2}$. Let $G$ be the quotient $E/F$, respectively the subobject of the map $E \onto F$. We have
\[
H \cdot \ch(G) = \left(0, 1, y - d - \frac{1}{2}, e - z - y + \frac{1}{6} \right).
\]
Lemma \ref{lem:bound_torsion_on_plane} implies
\begin{align*}
0 &\leq \frac{1}{24} + \frac{(y - d - \frac{1}{2})^2}{2} - \left(e - z - y + \frac{1}{6} \right) = \frac{(d-y)^2}{2} + \frac{d+y}{2} + z - e  \\
&< \frac{(d-y)^2}{2} + \frac{d+y}{2} + \frac{y(y+1)}{2} - \frac{d(d+1)}{2} = y^2 + y - dy \leq 0.
\end{align*}
The final inequality is obtained as follows: In $y$ the parabola $y^2 + y - dy$ has a minimum for 
\[
y_0 = \frac{d-1}{2} > y. 
\]
Henceforth, the maximum occurs for $y = 0$.
\end{proof}

\subsection{Stable rank zero objects}\label{subsec:rk0}

In this section, we will prove the following bound for rank zero objects. It is a generalization of Lemma \ref{lem:bound_torsion_on_plane} beyond objects supported on divisors numerically equivalent to $H$.

\begin{thm}
\label{thm:rank_zero_bound}
Let $E \in \Coh^{\beta}(X)$ be a $\nu_{\alpha, \beta}$-semistable object with $H \cdot \ch(E) = (0,c,d,e)$, where $c > 0$. Then 
\begin{equation}
\label{eq:rank_zero_inequality}
e \leq \frac{c^3}{24} + \frac{d^2}{2c} - \varepsilon\left(d + \frac{c^2}{2}, c\right).
\end{equation}
\end{thm}

As in the previous section, we will replace $\varepsilon$ by $\tilde{\varepsilon}$ and $E$ by $\tilde{E}$ in the proof since the difference is just a rounding term that comes for free at the end. The case $c = 1$ was already shown in Lemma \ref{lem:bound_torsion_on_plane}. Throughout this section, we will prove the theorem, assume its notation, and $c \geq 2$.
\begin{lem}
\label{lem:complicated_bounds_rank0}
Assume that (\ref{eq:rank_zero_inequality}) does not hold.
\begin{enumerate}[(i)]
\item
We can bound
\label{item:bound_epsilon}
\[
\tilde{\varepsilon}\left(d + \frac{c^2}{2},c\right) \leq \frac{c^2-c}{8}.
\]
\item
\label{item:bg_circle}
The radius $\rho_Q$ of the semidisk $Q_{\alpha, \beta}(E) \leq 0$ is given by
\[
\rho^2_Q = \frac{6ce - 3d^2}{c^2} > \frac{c^2}{4} - \frac{6\tilde{\varepsilon}\left(d + \frac{c^2}{2},c\right)}{c}\geq \frac{c^2 - 3c + 3}{4} > \left( \frac{2c-3}{4} \right)^2.
\]
\item 
\label{item:not_barbeta_stable}
The object $E$ is destabilized along a wall $W$ induced by $0 \to F \to E \to G \to 0$ or $0 \to G \to E \to F \to 0$, where $F$ has positive rank. Let $H \cdot \ch(F) = (r,x,y,z)$.
\item 
We have $r = 1$.
\item
\label{item:bound_y}
There are inequalities
\[
\frac{c^2}{8} + \frac{dx}{c} - \frac{d^2}{2c^2} - \frac{3f}{2} + \frac{3f^2}{2c} + \frac{3f}{2c} -\frac{3f^2}{2c^2} < y \leq \min \left\{ \frac{x^2}{2}, \frac{(c-x)^2}{2} + d \right\}.
\]
\item 
\label{item:bound_on_x}
The inequalities
\[
\frac{d}{c} + \frac{c}{2} + \frac{3}{4} \geq x \geq \frac{d}{c} + \frac{c}{2} - \frac{3}{4}
\]
hold, i.e., either $x = \tfrac{d}{c} + \tfrac{c}{2} + \frac{f}{c}$ or $x = \tfrac{d}{c} + \tfrac{c}{2} + \frac{f}{c} - 1$.
\item
\label{item:special_case_c_2}
If $c = 2$, then
\[
y = \min \left\{ \frac{x^2}{2}, \frac{(2-x)^2}{2} + d \right\}.
\]
\item
\label{item:ch_3_sub}
We have
\[
z \leq \frac{x^4}{8} - \frac{x^3}{3} - \frac{x^2y}{2} + \frac{x^2}{4} + xy + \frac{y^2}{2} - \frac{y}{2}.
\]
\item We have
\begin{align*}
\label{item:bound_e}
e \leq &\frac{c^4}{8} - \frac{c^3x}{2} + \frac{3c^2x^2}{4} - \frac{cx^3}{2} + \frac{x^4}{4} - \frac{c^3}{3} + \frac{c^2d}{2} + c^2x - cdx - cx^2 + \frac{dx^2}{2} - \frac{c^2y}{2} \\ &+ cxy - x^2y + \frac{c^2}{4} - cd + \frac{d^2}{2} - \frac{cx}{2} + dx + \frac{x^2}{2} + cy - dy + y^2 + \frac{d}{2} - y.
\end{align*}
\end{enumerate}
\end{lem}

\begin{proof}
\begin{enumerate}[(i)]
\item The function $\tilde{\varepsilon}(d + \tfrac{c^2}{2},c)$ has a maximum for $f = \tfrac{1}{2}c$ and in that case 
\[
\tilde{\varepsilon}(d + \tfrac{c^2}{2},c) = \frac{c^2-c}{8}.
\]
\item The semidisk $Q_{\alpha, \beta}(E) \leq 0$ is given by
\[
\alpha^2 + \left(\beta - \frac{d}{c}\right)^2  \leq \frac{6e - 3d^2}{c^2}.
\]
The fact that (\ref{eq:rank_zero_inequality}) does not hold and part (\ref{item:bound_epsilon}) lead to the following inequalities
\[
\frac{6ce - 3d^2}{c^2} > \frac{c^2}{4} - \frac{6\tilde{\varepsilon}\left(d + \frac{c^2}{2},c\right)}{c}\geq \frac{c^2 - 3c + 3}{4} > \left( \frac{2c-3}{4} \right)^2.
\]
\item By part (\ref{item:bg_circle}) the region $Q_{\alpha, \beta}(E) < 0$ is non empty and there has to be a wall for $E$.
\item When $E$ gets destabilized there is either a subobject or a quotient $F$ with non-negative rank. If $\ch_0(F) \geq 2$, then by Lemma \ref{lem:higherRankBound} the radius $\rho_W$ must satisfy
\[
\rho_W^2 \leq \frac{\Delta(E)}{4 \ch_0(F)^2} \leq \frac{c^2}{16} \leq \frac{c^2}{4} - \frac{3c}{4}  + \frac{3}{4}.
\]
But by part (\ref{item:bg_circle}) the wall $W$ is too small to exist. Assume that $\ch_0(F) = 0$. Then the equation $\nu_{\alpha, \beta}(F) = \nu_{\alpha, \beta}(E)$ is independent of $(\alpha, \beta)$ and cannot induce a wall.
\item Solving the two inequalities $\Delta(F), \Delta(G) \geq 0$ for y leads to
\[
y \leq \min \left\{ \frac{x^2}{2}, \frac{(c-x)^2}{2} + d \right\}.
\]
We know that $W$ must be outside the semidisk $Q_{\alpha, \beta}(E) < 0$. The radius of $W$ can be computed as
\[
\rho_W^2 = \frac{2c^2y + d^2 - 2cdx}{c^2} \geq \rho^2_Q \geq \frac{c^2}{4} - \frac{6\tilde{\varepsilon}(d + \tfrac{c^2}{2},c)}{c}.
\]
Solving the inequality for $y$ leads to the claimed lower bound for $y$.
\item We know that
\[
\rho_W \geq \frac{2c-3}{4}.
\]
and $s_W = \tfrac{d}{c}$. By construction of $\Coh^{\beta}(X)$, we have $0 \leq \frac{H^2\cdot\ch_1^{\beta}(F)}{H^3} \leq c$ for all $\beta$ appearing along the wall $W$. Rearranging the terms leads to
\[
c + \beta \geq x \geq \beta.
\]
Since the middle term is independent of $\beta$, we can vary $\beta$ independently on the left and right to get
\[
\frac{d}{c} + \frac{c}{2} + \frac{3}{4} \geq c + s_W - \rho_W \geq x \geq s_W + \rho_W \geq \frac{d}{c} + \frac{c}{2} - \frac{3}{4}.
\]
\item Assume $c = 2$. If $x = \tfrac{d}{c} + \tfrac{c}{2} + \tfrac{f}{c}$. Then
\[
\frac{(c-x)^2}{2} + d - \left(\frac{c^2}{8} + \frac{dx}{c} - \frac{d^2}{2c^2} - \frac{3f}{2} + \frac{3f^2}{2c} + \frac{3f}{2c} -\frac{3f^2}{2c^2} \right) = -\frac{f(f-1)}{4} < \frac{1}{2}.
\]
If $x = \tfrac{d}{c} + \tfrac{c}{2} + \tfrac{f}{c} - 1$. Then
\[
\frac{x^2}{2} - \left(\frac{c^2}{8} + \frac{dx}{c} - \frac{d^2}{2c^2} - \frac{3f}{2} + \frac{3f^2}{2c} + \frac{3f}{2c} -\frac{3f^2}{2c^2} \right) = -\frac{(f-1)(f-2)}{4} < \frac{1}{2}.
\]
In both cases, we can conclude by part (\ref{item:bound_y}).
\item Since $F$ has rank $1$, we can use Proposition \ref{prop:maximum_ch3_rank1} to bound $z$ as claimed.
\item Since $G$ has rank $-1$, Proposition \ref{prop:maximum_ch3_rank1} applies, and together with (\ref{item:ch_3_sub}) the last claim follows. \qedhere
\end{enumerate}
\end{proof}

\begin{proof}[Proof of Theorem \ref{thm:rank_zero_bound}]
We need to maximize the function
\begin{align*}
g_{c,d,x}(y) := &\frac{c^4}{8} - \frac{c^3x}{2} + \frac{3c^2x^2}{4} - \frac{cx^3}{2} + \frac{x^4}{4} - \frac{c^3}{3} + \frac{c^2d}{2} + c^2x - cdx - cx^2 + \frac{dx^2}{2} - \frac{c^2y}{2} \\ &+ cxy - x^2y + \frac{c^2}{4} - cd + \frac{d^2}{2} - \frac{cx}{2} + dx + \frac{x^2}{2} + cy - dy + y^2 + \frac{d}{2} - y
\end{align*}
under the numerical constraints imposed on $c,d,x,y$ by Lemma \ref{lem:complicated_bounds_rank0}. The argument works as follows: We will show that $g_{c,d,x}(y)$ is increasing in $y$ and therefore, reduce to 
\[
y = \min \left\{ \frac{x^2}{2}, \frac{(c-x)^2}{2} + d \right\}.
\]
Note that we can assume $d \geq 3$, since for $d = 2$ we established that this is the only possible value for $y$. The function $g_{c,d,x}(y)$ is a parabola in $y$ with minimum at
\[
y_0 = \frac{c^2}{4} - \frac{cx}{2} + \frac{x^2}{2} - \frac{c}{2} + \frac{d}{2} + \frac{1}{2}.
\]
The proof will proceed individually for each of the two possible values for $x$.
\begin{enumerate}
\item Assume that $x = \tfrac{d}{c} + \tfrac{c}{2} + \frac{f}{c}$. Then
\begin{align*}
y - y_0 &> \left(\frac{c^2}{8} + \frac{dx}{c} - \frac{d^2}{2c^2} - \frac{3f}{2} + \frac{3f^2}{2c} + \frac{3f}{2c} -\frac{3f^2}{2c^2}\right) - \left(\frac{c^2}{4} - \frac{cx}{2} + \frac{x^2}{2} - \frac{c}{2} + \frac{d}{2} + \frac{1}{2} \right) \\
&= \frac{c}{2} - \frac{3f}{2} + \frac{3f^2}{2c} + \frac{3f}{2c} - \frac{2f^2}{c^2} - \frac{1}{2} =: h_c(f).
\end{align*}
Note that $h_c(f)$ is a parabola in $f$ with minimum at
\[
f_0 = \frac{3c^2 - 3c}{6c - 8}.
\]
For $c \geq 3$ we get 
\[
h_c(f_0) = \frac{(3c - 7)(c - 1)}{24c - 32} > 0,
\]
and thus, $g_{c,d,x}(y)$ is increasing under our restrictions on $c,d,x,y$. Finally
\[
g_{c,d,x}\left( \frac{(c-x)^2}{2} + d \right) = \frac{c^3}{24} + \frac{d^2}{2c} - \tilde{\varepsilon}\left(d + \frac{c^2}{2}, c \right).
\]
\item Assume that $x = \tfrac{d}{c} + \tfrac{c}{2} + \frac{f}{c} - 1$. Then
\begin{align*}
y - y_0 &> \left(\frac{c^2}{8} + \frac{dx}{c} - \frac{d^2}{2c^2} - \frac{3f}{2} + \frac{3f^2}{2c} + \frac{3f}{2c} -\frac{3f^2}{2c^2}\right) - \left(\frac{c^2}{4} - \frac{cx}{2} + \frac{x^2}{2} - \frac{c}{2} + \frac{d}{2} + \frac{1}{2} \right) \\
&= \frac{c}{2} - \frac{3f}{2} + \frac{3f^2}{2c} + \frac{5f}{2c} - \frac{2f^2}{c^2} - 1 =: h_c(f).
\end{align*}
Note that $h_c(f)$ is a parabola in $f$ with minimum at
\[
f_0 = \frac{3c^2 - 5c}{6c - 8}.
\]
For $c \geq 3$ we get 
\[
h_c(f_0) =  \frac{(3c - 7)(c - 1)}{24c - 32} > 0,
\]
and thus, $g_{c,d,x}(y)$ is increasing under our restrictions on $c,d,x,y$. Finally
\[
g_{c,d,x}\left( \frac{x^2}{2} \right) = \frac{c^3}{24} + \frac{d^2}{2c} - \tilde{\varepsilon}\left(d + \frac{c^2}{2}, c \right). \qedhere
\]
\end{enumerate}
\end{proof}

\subsection{Stable rank two objects}\label{subsec:rk2}

In this section we prove the following bound for tilt stable rank two objects.

\begin{thm}
\label{thm:rank_2_first_cases}
Let $E \in \Coh(X)$ be a $\nu_{\alpha, \beta}$-semistable rank two object  for some $(\alpha, \beta)$ with $H \cdot \ch(E) = (2,c,d,e)$. 
\begin{enumerate}[(i)]
\item If $c = -1$, then $d \leq 0$ and 
\[
e \leq \frac{d^2}{2} - d + \frac{5}{24} - \varepsilon\left(d + \frac{1}{2},1\right).
\]
\item If $c = 0$, then $d \leq 0$.
\begin{enumerate}[(a)]
\item If $d = 0$, then $e \leq 0$.
\item If $d = -\tfrac{1}{2}$, then $e \leq \tfrac{1}{6}$.
\item If $d \leq -1$, then
\[
e \leq \frac{d^2}{2} + \frac{5}{24} - \varepsilon\left(d + \frac{1}{2},1\right).
\]
\end{enumerate}
\end{enumerate}
\end{thm}

Note that 
\[
\frac{d^2}{2} + \frac{5}{24} - \varepsilon\left(d + \frac{1}{2},1\right) \in \frac{1}{6} \Z
\]
for any $d \in \tfrac{1}{2} \Z$. Again, we will be able to ignore $\varepsilon\left(d + \frac{1}{2},1\right)$ as it is simply a rounding term. 

The bounds on $d$ are a consequence of the Bogomolov inequality. The bounds on $e$ will be proved via an induction on the discriminant $\Delta(E)$. We start with two lemmas for cases of low discriminant.

\begin{lem}
\label{lem:c_0_induction_start1}
If $E \in \Coh^{\beta}(X)$ is a tilt semistable object with $H \cdot \ch(E) = \left(2, 0, 0, e\right)$, then $e \leq 0$.
\end{lem}

\begin{proof}
Since $E \in \Coh^{\beta}(X)$, we must have $\beta < 0$. The inequality on $e$ is then equivalent to $Q_{\alpha, \beta}(E) \geq 0$.
\end{proof}

\begin{lem}
\label{lem:c_0_induction_start2}
If $E \in \Coh^{\beta}(X)$ is a tilt semistable object with $H \cdot \ch(E) = \left(2, 0, -\tfrac{1}{2}, e\right)$, then $e \leq \tfrac{1}{6}$.
\end{lem}

\begin{proof}
If $E$ is semistable for some $\alpha > 0$, $\beta \in \R$ inside the closed semidisk
\[
\alpha^2 + \left(\beta + \frac{3}{4} \right)^2 \leq \frac{1}{16},
\]
then $Q_{\alpha, \beta}(E) \geq 0$ implies $e \leq \tfrac{1}{4}$. By Assumption \ref{assu:div_curves}, this means $e \leq \tfrac{1}{6}$.

Next, we will show that $E$ has to be semistable for some point inside this closed semidisk. If not, then $E$ is destabilized by a semistable subobject $F$ along the vertical line $\beta = -1$. If $\ch_0(F) \leq 0$, then we work with the quotient $E/F$ instead. Therefore, we can assume $\ch_0(F) \geq 1$. We can compute
\[
H \cdot \ch_{\leq 2}^{-1}(E) = \left(2,2,\frac{1}{2} \right).
\]
Let $H \cdot \ch_{\leq 2}^{-1}(F) = (r,x,y)$. Then the definition of $\Coh^{\beta}(X)$ and the fact that we are not dealing with a vertical wall implies $0 < x < 2$, i.e., $x = 1$. At the wall, we have
\[
y - \frac{r}{2} \alpha^2 = \nu_{\alpha, -1}(F) = \nu_{\alpha, -1}(E) = \frac{1}{4} - \frac{1}{2}\alpha^2.
\]
If $r = 1$, then this implies $y = \tfrac{1}{4}$ in contradiction to Assumption \ref{assu:div_curves}. If $r \geq 2$, we get $y > \tfrac{1}{4}$. However, that implies $\Delta(F) = 1 - 2ry \leq 1 - 4y < 0$, a contradiction.
\end{proof}

\begin{lem}
\label{lem:no_rank_3_walls}
Let $E \in \Coh^{\beta}(X)$ be a tilt semistable object with $H \cdot \ch(E) = (2,c,d,e)$. Assume that either
\begin{enumerate}[(i)]
\item $c = -1$, $d \leq 0$, and $e \geq \tfrac{d^2}{2} - d + \tfrac{1}{3}$, or
\item $c = 0$, $d \leq -1$, and $e \geq \tfrac{d^2}{2} + \tfrac{1}{3}$.
\end{enumerate}
Then $E$ is destabilized along a semicircular wall induced by an exact sequence $0 \to F \to E \to G \to 0$, where $F$ and $G$ have rank at most $2$.
\end{lem}

\begin{proof}
\begin{enumerate}[(i)]
\item Assume $c = -1$. Then the radius $\rho_Q$ of the semidisk $Q_{\alpha, \beta}(E) \leq 0$ satisfies
\begin{align*}
\rho_Q^2 - \frac{\Delta(E)}{12} &= \frac{16d^3 - 3d^2 + 36de + 36e^2 - 6e}{(4d - 1)^2} + \frac{d}{3} - \frac{1}{12} \\
&\geq \frac{108d^4 + 40d^3 + 24d^2 - 60d + 23}{12(4d - 1)^2} > 0.
\end{align*}
By Lemma \ref{lem:higherRankBound} any destabilizing subobjects or quotients must have rank smaller than or equal to two.
\item Assume $c = 0$. Then the radius $\rho_Q$ of the semidisk $Q_{\alpha, \beta}(E) \leq 0$ satisfies
\begin{align*}
\rho_Q^2 - \frac{\Delta(E)}{12} &= \frac{4d^3 + 9e^2}{4d^2} + \frac{d}{3} \\
&\geq \frac{27d^4 + 64d^3 + 36d^2 + 12}{48d^2} > 0.
\end{align*}
By Lemma \ref{lem:higherRankBound} any destabilizing subobjects or quotients must have rank smaller than or equal to two. \qedhere
\end{enumerate}
\end{proof}

\begin{proof}[Proof of Theorem \ref{thm:rank_2_first_cases}]
If $c = -1$, then $\Delta(E) = 1 - 4d \geq 0$. Since $d \in \tfrac{1}{2} \Z$, we get $d \leq 0$. If $c = 0$, then $\Delta(E) = -4d \geq 0$ implies $d \leq 0$. We will prove the bounds on $e$ simultaneously in both cases via induction on $\Delta(E)$. Lemmas \ref{lem:c_0_induction_start1} and \ref{lem:c_0_induction_start2} 
provide the start of the induction. For a contradiction assume that the upper bounds on $e$ claimed in the theorem do not hold.

The strategy of the proof is to show that there is no wall outside the semidisk $Q_{\alpha, \beta}(E) < 0$ and therefore, there is no wall for such an object. By Lemma \ref{lem:no_rank_3_walls} we know that $E$ is destabilized along a semicircular wall $W$ induced by a subobject $F \into E$ of rank smaller than or equal to two. By replacing $F$ with the quotient $E/F$ if necessary, we can assume that $F$ has rank $1$ or $2$. Let $H \cdot \ch(F) = (r,x,y,z)$. Note that we have $\Delta(F) < \Delta(E)$, and we intend to use the induction hypothesis on $F$ in case $r = 2$.

\begin{enumerate}[(i)]
\item Assume $c = -1$, $d \leq 0$, and $e \geq \tfrac{d^2}{2} - d + \tfrac{1}{3}$.
\begin{enumerate}[(a)]
\item Assume $F$ has rank one. Then
\[
Q_{0, -\frac{3}{2}}(E) = 4d^2 - 12d - 12e + \frac{9}{4} \leq -2d^2 - \frac{7}{4} < 0.
\]
This implies
\[
x + \frac{3}{2} = \frac{H^2 \cdot \ch^{-\frac{3}{2}}_1(F)}{H^3} > 0.
\]
Since $x \geq 0$ implies the wall to be on the wrong side of the vertical wall, we must have $x = -1$. The Bogomolov inequality $\Delta(F) \geq 0$ implies $y \leq \tfrac{1}{2}$. We need a second bound of $y$ from below. We have $s(E, F) = d - 2y$. Moreover, the center of the semidisk $Q_{\alpha, \beta}(E) < 0$ is given by
\[
s_Q = \frac{d + 6e}{4d - 1}.
\]
Since no wall can be inside this semidisk, we must have
\begin{align*}
0 &\leq s_Q - s(E, F) = \frac{d + 6e}{4d - 1} - d + 2y \\
&\leq \frac{d^2 - 8dy + 4d + 2y - 2}{1 - 4d}.
\end{align*}
This implies,
\[
y \geq \frac{d^2 + 4d - 2}{8d - 2}.
\]
Applying Proposition \ref{prop:maximum_ch3_rank1} to $F$ implies
\[
z \leq \frac{y^2}{2} - 2y + \frac{17}{24}.
\]
We can apply the same proposition to the quotient $E/F$ to obtain
\begin{align*}
e &\leq \frac{d^2}{2} - dy + \frac{y^2}{2} - \frac{d}{2} + \frac{y}{2} + z \\
&\leq \frac{d^2}{2} - dy + y^2 - \frac{d}{2} - \frac{3y}{2} + \frac{17}{24} =: \varphi_d(y).
\end{align*}
We have to maximize this expression in $y$ which defines a parabola with minimum. Therefore, the maximum has to occur on the boundary. We can compute
\[
\varphi_d\left(\frac{1}{2}\right) = \frac{d^2}{2} - d + \frac{5}{24}.
\]
Moreover,
\[
\frac{d^2}{2} - d + \frac{1}{3} - \varphi_d\left( \frac{d^2 + 4d - 2}{8d - 2} \right) = \frac{14d^4 + 4d^3 + 2d^2 - 12d + 1}{8(4d - 1)^2} > 0.
\]
\item Assume $F$ has rank two. As in the rank one case, we get
\[
x + 3 = \frac{H^2 \cdot \ch_1^{-\frac{3}{2}}(F)}{H^3} > 0,
\]
i.e., $x \geq -2$. If $x = -1$, we are dealing with the vertical wall, and if $x \geq 0$, then the wall is on the wrong side of the vertical wall. Overall, we must have $x = -2$. The Bogomolov inequality says $y \leq 1$. We need to bound $y$ from below. We have $s(E, F) = d - y$. Moreover, the center of the semidisk $Q_{\alpha, \beta}(E) < 0$ is given by
\[
s_Q = \frac{d + 6e}{4d - 1}.
\]
Since no wall can be inside this semidisk, we must have
\begin{align*}
0 &\leq s_Q - s(E, F) = \frac{d + 6e}{4d - 1} - d + y \\
&\leq \frac{d^2 - 4dy + 4d + y - 2}{1 - 4d}.
\end{align*}
This implies,
\[
y \geq \frac{d^2 + 4d - 2}{4d - 1}.
\]
In particular, for $d = 0$, we have $1 \geq y \geq 2$, and such a wall simply cannot exist. Therefore, we can assume $d \leq -\tfrac{1}{2}$. Applying Lemma \ref{lem:bound_torsion_on_plane} to the quotient $E/F$ leads to
\[
e \leq \frac{(d-y)^2}{2} + z + \frac{1}{24}.
\]
The next step is to apply induction to $F(1)$. If $y = 1$, then $H\cdot\ch_3(F(1)) = z + \tfrac{1}{3} \leq 0$, i.e., $z \leq -\tfrac{1}{3}$. This leads to
\[
e \leq \frac{d^2}{2} - d + \frac{5}{24}.
\]
If $y = \tfrac{1}{2}$, then $\ch_3(F(1)) = z - \tfrac{1}{6} \leq \tfrac{1}{6}$, i.e., $z \leq \tfrac{1}{3}$. This leads to
\[
e \leq \frac{d^2}{2} - \frac{d}{2} + \frac{1}{2} < \frac{d^2}{2} - d + \frac{1}{3}.
\]
Assume $y \leq 0$. Then we have
\[
\frac{d^2 + 4d - 2}{4d - 1} \leq y \leq 0.
\]
This implies $d \leq -\tfrac{9}{2}$. We can apply induction to $F(1)$ to get 
\[
z \leq \frac{y^2}{2} - 2y + \frac{11}{8}.
\]
Therefore,
\[
e \leq \frac{d^2}{2} - dy + y^2 - 2y + \frac{17}{12} =: \varphi_d(y).
\]
This expression defines a parabola with minimum in $y$. Therefore its maximum will occur on the boundary. We have
\[
\frac{d^2}{2} - d + \frac{1}{3} - \varphi_d(0) \geq -d - \frac{13}{12} > 0.
\]
On the other boundary point, we have
\[
\frac{d^2}{2} - d + \frac{1}{3} - \varphi_d\left( \frac{d^2 + 4d - 2}{4d - 1} \right) = \frac{36d^4 - 12d^3 - 40d^2 + 20d - 13}{12(4d - 1)^2} > 0.
\]
\end{enumerate}
\item Assume $c = 0$, $d \leq -1$, and $e \geq \tfrac{1}{2}d^2 + \tfrac{1}{3}$.
\begin{enumerate}[(a)]
\item Assume that $F$ has rank one. Then
\[
Q_{0,-1}(E) = 4d^2 - 4d - 12e \leq -2d^2 - 4d - 4 < 0.
\]
Thus, $W(F,E)$ must contain a point $(\alpha, -1)$ and we have 
\[
0 < \frac{H^2 \cdot \ch^{-1}_1(F)}{H^3} = x + 1,
\]
i.e., $x > -1$. However, if $x \geq 0$, then we are either dealing with the vertical wall or a wall to the right of the vertical wall. Therefore, the wall cannot exist.
\item Assume that $F$ has rank two. Then as in the rank one case, we have $Q_{0,-1}(E) < 0$, and thus, 
\[
0 < \frac{H^2 \cdot \ch^{-1}_1(F)}{H^3} = x + 2,
\]
i.e., $x > -2$. Again, $x \geq 0$ implies that we do not deal with a wall to the left of the vertical wall. Overall, we must have $x = -1$. The inequality $\Delta(F) \geq 0$ implies $y \leq 0$. We need to bound $y$ from below. We have $s(E, F) = d - y$. Moreover, the center of the semidisk $Q_{\alpha, \beta}(E) < 0$ is given by
\[
s_Q = \frac{3e}{2d}.
\]
Since no wall can be inside this semidisk, we must have
\[
0 \leq s_Q - s(E, F) \leq -\frac{d^2 - 4dy - 2}{4d}.
\]
This implies
\[
y \geq \frac{d^2 - 2}{4d}.
\]
We can apply Lemma \ref{lem:bound_torsion_on_plane} to the quotient $E/F$ to get
\[
e \leq \frac{(d-y)^2}{2} + z + \frac{1}{24}.
\]
The next step is to apply induction to $F$ to get
\[
z \leq \frac{y^2}{2} - y + \frac{5}{24}
\]
which implies
\begin{align*}
e &\leq  \frac{(d-y)^2}{2} + z + \frac{1}{24} \leq \frac{d^2}{2} - dy + y^2 - y + \frac{1}{4} =: \varphi_d(y).
\end{align*}
This expression defines a parabola in $y$ with minimum. Therefore, the maximum will occur on a boundary point. We get
\[
\frac{d^2}{2} + \frac{1}{3} - \varphi_d(0) \geq \frac{1}{12} > 0.
\]
Moreover,
\[
\frac{d^2}{2} + \frac{1}{3} - \varphi_d\left(\frac{d^2 - 2}{4d}\right) = \frac{9d^4 + 12d^3 - 8d^2 - 24d - 12}{48d^2} > 0. \qedhere
\]
\end{enumerate}
\end{enumerate}
\end{proof}

\subsection{Bounding the arithmetic genus of integral curves}

We will now prove Theorem \ref{thm:genus_bound_large_degree} in a series of lemmas. The proof will be by contradiction. We already dealt with the case $k = 1$ in Proposition \ref{prop:maximum_ch3_rank1}. Therefore, we will assume $k \geq 2$ in this proof.

\begin{lem}
\label{lem:second_error_useless}
Under the assumptions of Theorem \ref{thm:genus_bound_large_degree} the inequality $e > E(d,k)$ implies $e \geq \tilde{E}(d,k)+ \tfrac{1}{8} > \tilde{E}(d,k)$. In particular, Theorem \ref{thm:genus_bound_large_degree} requires only to prove $e \leq \tilde{E}(d,k)$.
\end{lem}
\begin{proof}
We know $e \in \tfrac{1}{6} \Z$. The statement now follows from the fact that $E(d,k) \in \tfrac{1}{6} \Z$ and $\varepsilon(d,1) + \tfrac{1}{8} \leq \tfrac{1}{6}$.
\end{proof}

\begin{lem}
\label{lem:genus_bound_decreases}
\begin{enumerate}[(i)]
\item
\label{item:bound_epsilon_again}
The bound
\[
\tilde{\varepsilon}(d,k) \leq \frac{k^2}{8} - \frac{k}{8}
\]
holds.
\item If $h > k$ and $d > h(h-1)$, then $\tilde{E}(d,k) > \tilde{E}(d,h)$, i.e., the function $\tilde{E}(d,k)$ is strictly decreasing in $k$ as long as $d > k(k-1)$.
\end{enumerate}
\end{lem}

\begin{proof}
The first part was already observed in Lemma \ref{lem:complicated_bounds_rank0} part (i): $\tilde{\varepsilon}(d,k)$ is a parabola in $f$ with maximum at $f = \tfrac{k}{2}$.

For the second part it is enough to show $\tilde{E}(d,k) > \tilde{E}(d,k+1)$ whenever $d > k(k+1)$. Let $m \geq k+2$ be the unique integer such that $d = mk - f$, where $0 \leq f < k$. We will deal with the following three cases individually.
\begin{enumerate}[(a)]
\item Assume $m \geq 2k+1$. Then $d > 2k^2$. We have
\begin{align*}
\tilde{E}(d, k) - \tilde{E}(d, k+1) &\geq \left(\frac{d^2}{2k} + \frac{dk}{2} - \frac{k^2}{8} + \frac{k}{8} \right) - \left( \frac{d^2}{2(k+1)} + \frac{d(k+1)}{2} \right) \\ 
&= \frac{-k^4 - 4dk^2 + 4d^2 - 4dk + k^2}{8(k^2 + k)}.
\end{align*}
Therefore, it is enough to show that the function
\[
\varphi_k(d) = -k^4 - 4dk^2 + 4d^2 - 4dk + k^2
\]
is positive. The fact that $d > 2k^2$  implies $\varphi'_k(d) = -4k^2 + 8d - 4k > 0$. Thus, for $k \geq 2$, we have
\[
\varphi_k(d) \geq \varphi_k(2k^2) = 7k^4 - 8k^3 + k^2 > 0.
\]
If $k = 1$, then we have $d \geq k(k+1) + \tfrac{1}{2} = \tfrac{5}{2}$, and 
\[
\varphi_1\left(\frac{5}{2}\right) = 5 > 0.
\]
\item Assume $m \leq 2k$ and $0 \leq f \leq 2k - m + \tfrac{3}{2}$. Let $d \equiv -f' (\mod k+1)$ for $0 \leq f' < k + 1$. Then $f' = m + f - k - 1$ and
\[
\tilde{E}(d, k) - \tilde{E}(d, k+1) = (k - f)(m - k - 1) > 0.
\]
\item Assume $m \leq 2k$ and $2k - m + 2 \leq f < k$. Note, that this range for $f$ is non-empty if and only if $m \geq k + 3$. Let $d \equiv -f' (\mod k+1)$ for $0 \leq f' < k + 1$. Then $f' = m + f - 2k - 2$ and
\[
\tilde{E}(d, k) - \tilde{E}(d, k+1) = 2fk - 3k^2 - fm + 2km + f - 3k
\]
is linear in $f$ and increasing. Therefore, the minimum occurs at $f = 2k - m + 2$, where
\[
\tilde{E}(d, k) - \tilde{E}(d, k+1) = (m - k - 2)(m - k - 1) > 0. \qedhere
\]
\end{enumerate}
\end{proof}

\begin{lem}
\label{lem:sub_is_reflexive}
Let $C \subset X$ be an integral curve. Assume that $\II_C$ is destabilized via an exact sequence $0 \to E \to \II_C \to G \to 0$ in $\Coh^{\beta}(X)$ defining a semicircular wall in tilt stability. Then $E$ is a reflexive sheaf.
\end{lem}

\begin{proof}
The long exact sequences
\[
0 \to \HH^{-1}(E) \to 0 \to \HH^{-1}(G) \to \HH^0(E) = E \to \II_C \to \HH^0(G) \to 0
\]
shows that $E$ is a sheaf. Since both $\HH^{-1}(G)$ and $\II_C$ are torsion free sheaves, so is $E$. Let $Q$ be the cokernel of the natural inclusion $E \into E^{\vee \vee}$. If $Q = 0$, we are done.  If not, $Q$ has to be supported in dimension smaller than or equal to one. The strategy of the rest of the proof is to obtain a contradiction to $Q \neq 0$. We get a commutative diagram with exact sequences in $\Coh^{\beta}(X)$ as rows.

\[
\xymatrix{
  0 \ar[r] & E \ar[r] \ar[d] & E^{\vee \vee} \ar[r] \ar[d] & Q \ar[r] \ar[d] & 0 \\
  0 \ar[r] & \II_C \ar[r] & \II_C^{\vee \vee} = \OO_X \ar[r] & \OO_C \ar[r] & 0
}
\]

The kernel $K$ of $Q \to \OO_C$ in $\Coh^{\beta}(X)$ is also a torsion sheaf supported in dimension smaller than or equal to one. The Snake Lemma leads to a map $K \to G$, but $G$ is semistable and this map has to be trivial. Therefore, we get an injection $K \into E^{\vee \vee}$. Since $E^{\vee \vee}$ is torsion-free, we must have $K = 0$. Because $C$ is integral, the injective map $Q \to \OO_C$ implies that $Q$ is scheme-theoretically supported on $C$. Moreover, $E^{\vee \vee} \to \OO_X$ is injective in $\Coh^{\beta}(X)$.  Note that along the wall we must have $0 < H \cdot \ch_1^{\beta}(E) < H \cdot \ch_1^{\beta}(\II_C) = H \cdot\ch_1^{\beta}(\OO_X)$. Let $d$ be the degree of $C$. The final contradiction is obtained via
\begin{align*}
\nu_{\alpha, \beta}(E^{\vee \vee}) &= \frac{d}{H \cdot \ch_1^{\beta}(E)} + \nu_{\alpha, \beta}(E) = \frac{d}{H \cdot \ch_1^{\beta}(E)} + \nu_{\alpha, \beta}(\II_C) \\
& > \frac{d}{H \cdot \ch_1^{\beta}(\OO_X)} + \nu_{\alpha, \beta}(\II_C) = \nu_{\alpha, \beta}(\OO_X),
\end{align*}
because $\OO_X$ is stable in the whole $(\alpha, \beta)$-plane.
\end{proof}

\begin{lem}
\label{lem:BG_rank_bound}
The equation of the semidisk $Q_{\alpha, \beta}(\II_C) \leq 0$ is given by
\begin{align}
\label{eq:BG_rank_one}
\alpha^2 + \left(\beta + \frac{3e}{2d}\right)^2 = \frac{9e^2 - 8d^3}{4d^2}.
\end{align}
Assume $r$ is a positive integer such that
\[
e^2 > \frac{2(2r+1)^2}{9r(r+1)}d^3.
\]
Then $\II_C$ is destabilized via an exact sequence $0 \to E \to \II_C \to G \to 0$ defining a semicircular wall in tilt stability, where $0 < \ch_0(E) \leq r$.
\end{lem}

\begin{proof}
The claim about the equation is a straightforward calculation. For the second part observe that the hypothesis on $e^2$ implies
\[
\frac{9e^2 - 8d^3}{4d^2} > \frac{d}{2r(r+1)} > 0.
\]
Therefore, $\II_C$ has to destabilize at some point before $Q_{\alpha, \beta}(\II_C) < 0$. Moreover, we can use Lemma \ref{lem:higherRankBound} to conclude that $E$ cannot have rank greater than or equal to $r+1$.
\end{proof}

\begin{lem}
\label{lem:complicated_bounds_rank1}
Assume that $e > \tilde{E}(d,k)$ holds.
\begin{enumerate}[(i)]
\item 
\label{item:maximal_rank_of_walls}
The object $\II_C$ is destabilized along a semicircular wall $W$ induced by $0 \to E \to \II_C \to G \to 0$, where $E$ is a reflexive sheaf. Let $H \cdot \ch(E) = (r,x,y,z)$. Then either $r = 1$ or $r = 2$.
\item 
\label{item:bounding_ch1}
Let $\mu_0 \geq k$ be the unique integer such that either $d = \mu_0^2 - f'$ or $d = \mu_0(\mu_0+1) - f'$, where $d \equiv -f' (\mod \mu_0)$ for $0 \leq f' < \mu_0$. If $r = 1$, then $k \leq -x \leq \mu_0$. If $r = 2$, there are two possibilities.
\begin{enumerate}[(a)]
\item If $d = \mu_0^2 - f'$, then $x = -2\mu_0$ or $x = -2\mu_0 + 1$.
\item If $d = \mu_0(\mu_0+1) - f'$, then $x = -2\mu_0$ or $x = -2\mu_0 - 1$.
\end{enumerate}
\end{enumerate}
\end{lem}

\begin{proof}
\begin{enumerate}[(i)]
\item If such a wall exists, then reflexivity of $E$ follows immediately from Lemma \ref{lem:sub_is_reflexive}. By Lemma \ref{lem:BG_rank_bound} it is enough to show
\[
e > \frac{5}{3\sqrt{3}}d^{\frac{3}{2}}
\]
to get both $r \leq 2$ and existence of the wall. Because of Lemma \ref{lem:genus_bound_decreases} part (\ref{item:bound_epsilon_again}), and Lemma \ref{lem:second_error_useless} it is enough to show that the function
\[
\varphi_k(d) = \frac{d^2}{2k} + \frac{dk}{2} - \frac{k^2}{8} + \frac{k}{8} - \frac{5}{3\sqrt{3}}d^{\frac{3}{2}}
\]
is non-negative whenever $d > k(k-1)$. The function $\varphi_k$ has a local maximum at $d_0 = \tfrac{1}{3}k^2$, a local minimum at $d_1 = \tfrac{3}{4}k^2$, and no other local extrema. Since $d_0 < k(k-1)$, it is enough to check positivity at $d_1$. Indeed, we have
\[
\varphi_k(\frac{3}{4}k^2) = \frac{k(k-2)^2}{32} \geq 0.
\]
\item If $r = 1$, then $E$ is a line bundle, and $k \leq -x$ follows from $\Hom(\OO(-k+1), \II_C) = 0$. By Lemma \ref{lem:genus_bound_decreases}, we have $\tilde{E}(d,\mu_0) \leq \tilde{E}(d,k)$ and henceforth, $e > \tilde{E}(d,\mu_0)$. For any point $(\alpha, \beta) \in W$ with $\alpha > 0$, we have $H \cdot \ch_1^{\beta}(\II_C) > H \cdot \ch_1^{\beta}(E) > 0$. This can be rewritten as
\[
(r-1) \beta > x > r \beta.
\]
Note, that if $Q_{0, \beta}(\II_C) < 0$, then there is $\alpha > 0$ such that $(\alpha, \beta) \in W$. Since $x$ is an integer, the remaining claim is equivalent to showing the following two statements.
\begin{enumerate}[(a)]
\item If $d = \mu_0^2 - f'$, then $Q_{0, -\mu_0 - \tfrac{1}{2}}(\II_C) < 0$ and $Q_{0, 2-2\mu_0}(\II_C) < 0$.
\item If $d = \mu_0(\mu_0 + 1) - f'$, then $Q_{0, -\mu_0 - 1}(\II_C) < 0$ and $Q_{0, 1-2\mu_0}(\II_C) < 0$.
\end{enumerate}

The semidisk $Q_{0, -\mu_0 - 1}(\II_C) < 0$ becomes smaller when $e$ decreases. Therefore, it is enough to check these inequalities for $e = \tilde{E}(d,\mu_0)$.

Assume $d = \mu_0^2 - f'$. Then
\[
Q_{0, -\mu_0 - \tfrac{1}{2}}(\II_C) = -3f'^2\mu_0 + 2f'\mu_0^2 - \mu_0^3 + \frac{5}{2}f'^2 + f'\mu_0 + \frac{1}{2}\mu_0^2 - 2f'.
\]
This is a parabola in $f'$ with maximum at
\[
f' = \frac{2\mu_0^2 + \mu_0 - 2}{6\mu_0 - 5}.
\]
From this maximum, we get
\[
Q_{0, -\mu_0 - \tfrac{1}{2}}(\II_C) \leq -\frac{2(2\mu_0 + 1)(\mu_0 - 1)^3}{6\mu_0 - 5} < 0.
\]
Next, we have
\[
Q_{0, 2 - 2\mu_0}(\II_C) = -6f'^2\mu_0 + 8f'\mu_0^2 - 4\mu_0^3 + 10f'^2 - 14f'\mu_0 + 8\mu_0^2 - 2f'.
\]
This is a parabola in $f'$ with maximum at
\[
f' = \frac{4\mu_0^2 - 7\mu_0 - 1}{6\mu_0 - 10}.
\]
From this maximum, we get
\[
Q_{0, 2 - 2\mu_0}(\II_C) \leq -\frac{(8\mu_0^2 - 16\mu_0 - 1)(\mu_0 - 1)^2}{6\mu_0 - 10} < 0,
\]
unless $\mu_0 = 2$. In that case, we have $Q_{0, 2 - 2\mu_0}(\II_C) = -2f'^2 + 2f'$. However, we used $e = \tilde{E}(d,m)$, but could have used $e = \tilde{E}(d,m) + \frac{1}{8}$ to obtain a strict inequality.

Assume $d = \mu_0(\mu_0+1) - f'$. Then 
\[
Q_{0, -\mu_0 - 1}(\II_C) = -3f'^2\mu_0 + 2f'\mu_0^2 - \mu_0^3 + f'^2 + 3f'\mu_0 - 2\mu_0^2 + f' - \mu_0.
\]
This is a parabola in $f'$ with maximum at
\[
f' = \frac{2\mu_0^2 + 3\mu_0 + 1}{6\mu_0 - 2}.
\]
From this maximum, we get
\[
Q_{0, -\mu_0 - 1}(\II_C) \leq -\frac{(8\mu_0^2 - 8\mu_0 - 1)(\mu_0 + 1)^2}{12\mu_0 - 4} < 0.
\]
Finally, we have
\[
Q_{0, 1-2\mu_0}(\II_C) = -6f'^2\mu_0 + 8f'\mu_0^2 - 4\mu_0^3 + 7f'^2 - 6f'\mu_0 + \mu_0^2 - 5f' + 5\mu_0.
\]
This is a parabola in $f'$ with maximum at
\[
f' = \frac{8\mu_0^2 - 6\mu_0 - 5}{12\mu_0 - 14}.
\]
From this maximum, we get
\[
Q_{0, 1-2\mu_0}(\II_C) \leq -\frac{(4\mu_0^2 + 2\mu_0 - 5)(4\mu_0 - 5)(2\mu_0 - 1)}{24\mu_0 - 28} < 0. \qedhere
\]
\end{enumerate}
\end{proof}

\begin{lem}
\label{lem:rank2_wall_1}
Assume $d = \mu_0^2 - f'$, $d \equiv -f' (\mod \mu_0)$, $0 \leq f' < \mu_0$, and 
\[
e > \tilde{E}(d,\mu_0) = \mu_0^3 + \frac{1}{2}f'^2 - 2f'\mu_0 + \frac{1}{2}f'.
\]
Furthermore, suppose $\II_C$ is destabilized at a wall $W$ induced by an exact sequence $0 \to E \to \II_C \to G \to 0$, where $E$ is reflexive, and $E$ is tilt stable along $W$. 
\begin{enumerate}[(i)]
\item If $H\cdot\ch(E) = (2, -2\mu_0, y, z)$, then the following holds.
\begin{enumerate}[(a)]
\item The Chern character of $E(\mu_0)$ is $\left(2, 0, -\mu_0^2 + y, -\frac{2}{3}\mu_0^3 + \mu_0y + z\right)$.
\item We have
\[
\frac{2\mu_0^4 - 4f'\mu_0^2 + 3f'^2\mu_0 + 3f'\mu_0 - 4f'^2}{2(\mu_0^2 - f')} < y \leq \mu_0^2
\]
and
\[
0 \leq f' < \frac{2\mu_0^2 - 3\mu_0}{3\mu_0 - 4}.
\]
Moreover, if $\mu_0 \in \{2,3,4\}$, then $y = \mu_0^2$.
\item The inequality
\begin{align*}
e \leq &\frac{1}{2}\mu_0^4 + f'\mu_0^2 - \frac{2}{3}\mu_0^3 - \mu_0^2y + \frac{1}{2}f'^2 - 2f'\mu_0 + \\ &\frac{1}{2}\mu_0^2 - f'y + 2\mu_0y + \frac{1}{2}y^2 + \frac{1}{2}f' - \frac{1}{2}y + z
\end{align*}
holds.
\item If $y = \mu_0^2$, then
\begin{align*}
z &\leq -\frac{1}{3}\mu_0^3, \\
e &\leq \mu_0^3 + \frac{1}{2}f'^2 - 2f'\mu_0 + \frac{1}{2}f' = \tilde{E}(d,\mu_0).
\end{align*}
In particular, the wall $W$ cannot exist.
\item If $y = \mu_0^2 - \tfrac{1}{2}$, then
\begin{align*}
z &\leq -\frac{1}{3}\mu_0^3 + \frac{1}{2}\mu_0 + \frac{1}{6}, \\
e &\leq \mu_0^3 + \frac{1}{2}f'^2 - 2f'\mu_0 + f' - \frac{1}{2}\mu_0 + \frac{13}{24}.
\end{align*}
In particular, the wall $W$ cannot exist.
\item If $y \leq \mu_0^2 - 1$, then
\begin{align*}
z &\leq \frac{1}{2}\mu_0^4 + \frac{2}{3}\mu_0^3 - \mu_0^2y - \mu_0y + \frac{1}{2}y^2 + \frac{5}{24}, \\
e &\leq \mu_0^4 + f'\mu_0^2 - 2\mu_0^2y + \frac{1}{2}f'^2 - 2f'\mu_0 + \frac{1}{2}\mu_0^2 - f'y + \mu_0y + y^2 + \frac{1}{2}f' - \frac{1}{2}y + \frac{5}{24}.
\end{align*}
In particular, the wall $W$ cannot exist.
\end{enumerate}
\item If $H\cdot\ch(E) = (2, -2\mu_0 + 1, y, z)$, then the following holds.
\begin{enumerate}[(a)]
\item The Chern character of $E(\mu_0 - 1)$ is given by 
\[
\left(2, -1, -\mu_0^2 + \mu_0 + y, -\frac{2}{3}\mu_0^3 + \frac{3}{2}\mu_0^2 + \mu_0y - \mu_0 - y + z + \frac{1}{6}\right).
\]
\item The second Chern character of $E$ satisfies
\[
\frac{4\mu_0^4 - 8f'\mu_0^2 - 6\mu_0^3 - 11f'^2 + 6f'^2\mu_0 + 18f'\mu_0 - 3f'}{4(\mu_0^2 - f')} < y \leq \mu_0^2 + f' - 2\mu_0 + \frac{1}{2}.
\]
Moreover, in the case $\mu_0 = 2$ we have $f' \geq 1$.
\item The inequality
\begin{align*}
e \leq &\frac{1}{2}\mu_0^4 + f'\mu_0^2 - \frac{8}{3}\mu_0^3 - \mu_0^2y + \frac{1}{2}f'^2 - 4f'\mu_0 + 6\mu_0^2 \\ &- f'y + 4\mu_0y + \frac{1}{2}y^2 + 2f' - 4\mu_0 - 2y + z + \frac{17}{24}
\end{align*}
holds.
\item 
We have
\begin{align*}
z \leq &\frac{1}{2}\mu_0^4 - \frac{1}{3}\mu_0^3 - \mu_0^2y + \frac{1}{2}y^2 + \frac{1}{24}, \\
e \leq &\mu_0^4 + f'\mu_0^2 - 3\mu_0^3 - 2\mu_0^2y + \frac{1}{2}f'^2 - 4f'\mu_0 + 6\mu_0^2 \\ &- f'y + 4\mu_0y + y^2 + 2f' - 4\mu_0 - 2y + \frac{3}{4}.
\end{align*}
In particular, the wall $W$ cannot exist.
\end{enumerate}
\end{enumerate}
\end{lem}

\begin{proof}
\begin{enumerate}[(i)]
\item Assume $H\cdot\ch(E) = (2, -2\mu_0, y, z)$.
\begin{enumerate}[(a)]
\item This part is a straightforward calculation.
\item By Theorem \ref{thm:rank_2_first_cases}, we get $-\mu_0^2 + y \leq 0$, i.e., $y \leq \mu_0^2$. The center of $W(E, \II_C)$ is given by 
\[
s(E, \II_C) = \frac{-2\mu_0^2 + 2f' - y}{2\mu_0}
\]
The semidisk $Q_{\alpha, \beta}(\II_C) < 0$ becomes smaller when $e$ decreases, and we can bound it from below by setting $e = \tilde{E}(d, \mu_0)$. Therefore, the center $s_Q$ of this semidisk satisfies
\[
s_Q < \frac{-6\mu_0^3 - 3f'^2 + 12f'\mu_0 - 3f}{4\mu_0^2 - 4f'}.
\]
The lower bound on $y$ is a consequence of the fact that $W$ is outside of $Q_{\alpha, \beta}(\II_C) < 0$, i.e., $s(E, \II_C) \leq s_Q$. Comparing the lower and upper bound for $y$ leads to the bound on $f'$. If $\mu_0 \in \{2,3,4\}$, then
\[
\frac{2\mu_0^4 - 4f'\mu_0^2 + 3f'^2\mu_0 + 3f'\mu_0 - 4f'^2}{2(\mu_0^2 - f')} > \mu_0^2 - \frac{1}{2}.
\]
\item This inequality on $e$ is an application of Proposition \ref{prop:maximum_ch3_rank1} to the quotient $G$.
\item Assume $y = \mu_0^2$. The upper bound on $z$ follows from Theorem \ref{thm:rank_2_first_cases} applied to $E(\mu_0)$. The bound on $e$ is a direct consequence of applying the bound on $z$ to the previous upper bound for $e$.
\item Assume $y = \mu_0^2 - \tfrac{1}{2}$. The upper bound on $z$ follows from Theorem \ref{thm:rank_2_first_cases} applied to $E(\mu_0)$. The bound on $e$ is a direct consequence of applying the bound on $z$ to the previous upper bound for $e$. Indeed, for $\mu_0 \neq 2$ we can use
\[
f' < \frac{2\mu_0^2 - 3\mu_0}{3\mu_0 - 4}
\]
to get
\[
\mu_0^3 + \frac{1}{2}f'^2 - 2f'\mu_0 + f' - \frac{1}{2}\mu_0 + \frac{13}{24} \leq \mu_0^3 + \frac{1}{2}f'^2 - 2f'\mu_0 + \frac{1}{2}f'.
\]
\item Assume $y \leq \mu_0^2 - 1$. The upper bound on $z$ follows from Theorem \ref{thm:rank_2_first_cases} applied to $E(\mu_0)$. The bound on $e$ is a direct consequence of applying the bound on $z$ to the previous upper bound for $e$. We are left to show that this inequality implies $e \leq \tilde{E}(d, \mu_0)$. This can be done by showing that the function
\begin{align*}
&\tilde{E}(d, \mu_0) - \left(\mu_0^4 + f'\mu_0^2 - 2\mu_0^2y + \frac{1}{2}f'^2 - 2f'\mu_0 + \frac{1}{2}\mu_0^2 - f'y + \mu_0y + y^2 + \frac{1}{2}f' - \frac{1}{2}y + \frac{5}{24}\right) \\
&= -\mu_0^4 - f'\mu_0^2 + \mu_0^3 + 2\mu_0^2y - \frac{1}{2}\mu_0^2 + f'y - \mu_0y - y^2 + \frac{1}{2}y - \frac{5}{24} =: \varphi(y, f', \mu_0)
\end{align*} 
is non-negative. This function is a parabola in $y$ with maximum. The maximum occurs at
\[
y_0 = \mu_0^2 + \frac{1}{2}f' - \frac{1}{2}\mu_0 + \frac{1}{4}.
\]
We will show that $y_0$ lies before our range for $y$ and therefore, $\varphi_{f', \mu_0}$ has a minimum for $y = \mu^2 - 1$. Indeed, we can compute
\begin{align*}
\psi(\mu_0, f') &:= \frac{2\mu_0^4 - 4f'\mu_0^2 + 3f'^2\mu_0 + 3f'\mu_0 - 4f'^2}{2(\mu_0^2 - f')} - y_0 \\
&= \frac{6f'^2\mu_0 - 6f'\mu_0^2 + 2\mu_0^3 - 6f'^2 + 4f'\mu_0 - \mu_0^2 + f'}{4(\mu_0^2 - f')}.
\end{align*}
The numerator is a parabola in $f'$ with minimum at
\[
f'_0 = \frac{6\mu_0^2 - 4\mu_0 - 1}{12(\mu_0-1)}
\]
and
\[
\psi(\mu_0, f'_0) = \frac{12\mu_0^4 - 24\mu_0^3 + 20\mu_0^2 - 8\mu_0 - 1}{8(6\mu_0^2 - 6\mu_0 - 1)(2\mu_0 - 1)} > 0.
\]
Finally, we get
\[
\varphi(y, f', \mu_0) \geq \varphi(\mu_0^2 - 1, f', \mu_0) = -f' + \mu_0 - \frac{41}{24} \geq 0,
\]
due to the upper bound on $f'$ and $\mu_0 \geq 5$.
\end{enumerate}
\item Assume $H\cdot\ch(E) = (2, -2\mu_0 + 1, y, z)$.
\begin{enumerate}[(a)]
\item This part is a straightforward calculation.
\item The inequality $\Delta(G) \geq 0$ is equivalent to $y \leq \mu_0^2 + f' - 2\mu_0 + \tfrac{1}{2}$. 
The center of $W(E, \II_C)$ is given by 
\[
s(E, \II_C) = \frac{-2\mu_0^2 + 2f' - y}{2\mu_0 - 1}.
\]
The semidisk $Q_{\alpha, \beta}(\II_C) < 0$ becomes smaller when $e$ decreases, and we can bound it from below by setting $e = \tilde{E}(d, \mu_0)$. Therefore, the center $s_Q$ of this semidisk satisfies
\[
s_Q < \frac{-6\mu_0^3 - 3f'^2 + 12f'\mu_0 - 3f}{4\mu_0^2 - 4f'}.
\]
The lower bound on $y$ is then obtained from the fact that $W$ is outside this lower bound of the semidisk, i.e., $s(E, \II_C) < s_Q$. 
Finally, if $\mu_0 = 2$ then comparing the upper and lower bound on $y$ implies $f' \geq 1$.
\item This inequality on $e$ is a direct application of Proposition \ref{prop:maximum_ch3_rank1} to the quotient $G$.
\item 
The upper bound on $z$ follows from Theorem \ref{thm:rank_2_first_cases} applied to $E(\mu_0 - 1)$. The bound on $e$ is a direct consequence of estimating $z$ in the previous upper bound for $e$. We have to maximize the function
\begin{align*}
\varphi(y, f', \mu_0) := &\ \mu_0^4 + f'\mu_0^2 - 3\mu_0^3 - 2\mu_0^2y + \frac{1}{2}f'^2 - 4f'\mu_0 + 6\mu_0^2 \\ &- f'y + 4\mu_0y + y^2 + 2f' - 4\mu_0 - 2y + \frac{3}{4}.
\end{align*}
This function is a parabola in $y$ with minimum at
\[
y_0 = \mu_0^2 + \frac{1}{2}f' - 2\mu_0 + 1.
\]
We will show that $y_0$ lies before our range for $y$ and therefore, the maximum occurs at $y = \mu_0^2 + f' - 2\mu_0 + \tfrac{1}{2}$. Let
\begin{align*}
\psi(\mu_0, f') &:= \frac{4\mu_0^4 - 8f'\mu_0^2 - 6\mu_0^3 - 11f'^2 + 6f'^2\mu_0 + 18f'\mu_0 - 3f'}{4(\mu_0^2 - f')} - y_0 \\
&= \frac{6f'^2\mu_0 - 6f'\mu_0^2 + 2\mu_0^3 - 9f'^2 + 10f'\mu_0 - 4\mu_0^2 + f'}{4(\mu_0^2 - f')}.
\end{align*}
The numerator is a parabola in $f'$ with minimum at
\[
f'_0 = \frac{6\mu_0^2 - 10\mu_0 - 1}{12\mu_0 - 18}
\]
and
\[
\psi(\mu_0, f'_0) = \frac{12\mu_0^4 - 48\mu_0^3 + 56\mu_0^2 - 20\mu_0 - 1}{8(12\mu_0^3 - 24\mu_0^2 + 10\mu_0 + 1)},
\]
which is positive for $\mu_0 \geq 3$. When $\mu_0 = 2$, then we get $f'_0 = \tfrac{1}{2}$, but $f' \geq 1$. Then,
\[
\psi(2, 1) = 0.
\]
Finally, we get
\[
\varphi(y, f', \mu_0) \geq \varphi(\mu_0^2 + f' - 2\mu_0 + \tfrac{1}{2}, f', \mu_0) = \tilde{E}(d, \mu_0). \qedhere
\]
\end{enumerate}
\end{enumerate}
\end{proof}

\begin{lem}
\label{lem:rank2_wall_2}
Assume $d = \mu_0(\mu_0 + 1) - f'$, $d \equiv -f' (\mod \mu_0)$, $0 \leq f' < \mu_0$, and
\[
e > \tilde{E}(d,\mu_0) = \mu_0^3 + \frac{1}{2}f'^2 - 2f'\mu_0 + \frac{3}{2}\mu_0^2 - \frac{1}{2}f' + \frac{1}{2}\mu_0.
\]
Furthermore, suppose $\II_C$ is destabilized at a wall $W$ induced by an exact sequence $0 \to E \to \II_C \to G \to 0$, where $E$ is reflexive, and $E$ is tilt stable along $W$.
\begin{enumerate}[(i)]
\item If $H\cdot\ch(E) = (2, -2\mu_0, y, z)$, then the following holds.
\begin{enumerate}[(a)]
\item The Chern character of $E(\mu_0)$ is given by  $\left(2, 0, -\mu_0^2 + y, -\frac{2}{3}\mu_0^3 + \mu_0y + z\right)$.
\item We have
\[
\frac{2\mu_0^4 + 3f'^2\mu_0 - 4f'\mu_0^2 + \mu_0^3 - 4f'^2 + 5f\mu_0 - \mu_0^2}{2(\mu_0^2 - f' + \mu_0)} < y \leq \mu_0^2 + f' - \mu_0 < \mu_0^2.
\]
If $y = \mu_0^2 - \tfrac{1}{2}$, then $f' = \mu_0 - \tfrac{1}{2}$.
\item The inequality
\begin{align*}
e \leq &\frac{1}{2}\mu_0^4 + f'\mu_0^2 - \frac{5}{3}\mu_0^3 - \mu_0^2y + \frac{1}{2}f'^2 - 3f'\mu_0 + 3\mu_0^2 \\ &- f'y + 3\mu_0y + \frac{1}{2}y^2 + \frac{1}{2}f' - \frac{1}{2}\mu_0 - \frac{1}{2}y + z.
\end{align*}
holds.
\item If $y = \mu_0^2 - \tfrac{1}{2}$, then $f' = \mu_0 - \tfrac{1}{2}$ and
\begin{align*}
z &\leq -\frac{1}{3}\mu_0^3 + \frac{1}{2}\mu_0^2 + \frac{1}{6}, \\
e &\leq \mu_0^3 + \frac{1}{2}\mu_0^2 + \frac{1}{6}.
\end{align*}
In particular, the wall $W$ cannot exist.
\item If $y \leq \mu_0^2 - 1$, then
\begin{align*}
z &\leq \frac{1}{2}\mu_0^4 + \frac{2}{3}\mu_0^3 - \mu_0^2y - \mu_0y + \frac{1}{2}y^2 + \frac{5}{24}, \\
e &\leq \mu_0^4 + f'\mu_0^2 - \mu_0^3 - 2\mu_0^2y + \frac{1}{2}f'^2 - 3f'\mu_0 + 3\mu_0^2 - f'y + 2\mu_0y + y^2 + \frac{1}{2}f' - \frac{1}{2}\mu_0 - \frac{1}{2}y + \frac{5}{24}.
\end{align*}
In particular, the wall $W$ cannot exist.
\end{enumerate}
\item If $H\cdot\ch(E) = (2, -2\mu_0 - 1, y, z)$, then the following holds.
\begin{enumerate}[(a)]
\item The Chern character of $E(\mu_0)$ is given by $\left(2, -1, -\mu_0^2 - \mu_0 + y, -\frac{2}{3}\mu_0^3 - \frac{1}{2}\mu_0^2 + \mu_0y + z\right)$.
\item We have
\[
\frac{4\mu_0^4 + 6f'^2\mu_0 - 8f'\mu_0^2 + 8\mu_0^3 - 5f'^2 - 2f'\mu_0 + 7\mu_0^2 - 3f' + 3\mu_0}{4(\mu_0^2 - f' + \mu_0)} < y \leq \mu_0^2 + \mu_0.
\]
\item We have
\[
e \leq \frac{1}{2}\mu_0^4 + f'\mu_0^2 + \frac{1}{3}\mu_0^3 - \mu_0^2y + \frac{1}{2}f'^2 - f'\mu_0 + \frac{1}{2}\mu_0^2 - f'y + \mu_0y + \frac{1}{2}y^2 + z + \frac{1}{24}.
\]
\item We have
\begin{align*}
z &\leq \frac{1}{2}\mu_0^4 + \frac{5}{3}\mu_0^3 - \mu_0^2y + 2\mu_0^2 - 2\mu_0y + \frac{1}{2}y^2 + \mu_0 - y + \frac{5}{24}, \\
e &\leq \mu_0^4 + f'\mu_0^2 + 2\mu_0^3 - 2\mu_0^2y + \frac{1}{2}f'^2 - f'\mu_0 + \frac{5}{2}\mu_0^2 - f'y - \mu_0y + y^2 + \mu_0 - y + \frac{1}{4}.
\end{align*}
In particular, the wall $W$ can not exist.
\end{enumerate}
\end{enumerate}
\end{lem}

\begin{proof}
\begin{enumerate}[(i)]
\item Assume $H\cdot\ch(E) = (2, -2\mu_0, y, z)$.
\begin{enumerate}[(a)]
\item This part is a straightforward calculation.
\item The inequality $\Delta(G) \geq 0$ is equivalent to $y \leq \mu_0^2 + f' - \mu_0$. The semicircular wall $W(E, \II_C)$ has center
\[
s(E, \II_C) = - \frac{2\mu_0^2 - 2f' + 2\mu_0 + y}{2\mu_0},
\]
and the open semidisk $Q_{\alpha, \beta}(\II_C) < 0$ has center
\[
s_Q = -\frac{3e}{2(\mu_0^2 + \mu_0 - f')}.
\]
Since this semidisks center increases when $e$ decreases, the semidisk itself becomes smaller in that case, and we can bound it from below by setting $e = \tilde{E}(d, \mu_0)$. The lower bound on $y$ is then obtained from the fact that $W$ is outside this semidisk, i.e., $s(E, \II_C) \leq s_Q$.
If $y = \mu_0^2 - \tfrac{1}{2}$, then $y \leq \mu_0^2 + f' - \mu_0$ implies $f' = \mu_0 - \tfrac{1}{2}$.
\item This inequality on $e$ is a direct application of Proposition \ref{prop:maximum_ch3_rank1} to the quotient $G$.
\item Assume $y = \mu_0^2 - \tfrac{1}{2}$ and $f' = \mu_0 - \tfrac{1}{2}$. The upper bound on $z$ follows from Theorem \ref{thm:rank_2_first_cases} applied to $E(\mu_0)$. The bound on $e$ is a direct consequence of estimating $z$ in the previous upper bound for $e$. Indeed,
\[
e \leq \mu_0^3 + \frac{1}{2}\mu_0^2 + \frac{1}{6} < \mu_0^3 + \frac{1}{2}\mu_0^2 + \frac{3}{8} = \tilde{E}(d,\mu_0).
\]
\item Assume $y \leq \mu_0^2 - 1$. The upper bound on $z$ follows from Theorem \ref{thm:rank_2_first_cases} applied to $E(\mu_0)$. The bound on $e$ is a direct consequence of estimating $z$ in the previous upper bound for $e$. We are left to show that this inequality implies $e \leq \tilde{E}(d, \mu_0)$. This can be done by showing that the function
\begin{align*}
\varphi(y, f', \mu_0) = &\ \tilde{E}(d, \mu_0) - \left(\mu_0^4 + f'\mu_0^2 - \mu_0^3 - 2\mu_0^2y + \frac{1}{2}f'^2 - 3f'\mu_0 \right. \\ & \left. + 3\mu_0^2 - f'y + 2\mu_0y + y^2 + \frac{1}{2}f' - \frac{1}{2}\mu_0 - \frac{1}{2}y + \frac{5}{24}\right) \\
= &-\mu_0^4 - f'\mu_0^2 + 2\mu_0^3 + 2\mu_0^2y + f'\mu_0 - \frac{3}{2}\mu_0^2  + f'y \\  &\ - 2\mu_0y - y^2 - f' + \mu_0 + \frac{1}{2}y - \frac{5}{24}.
\end{align*}
is positive. This function is a parabola in $y$ with maximum. The maximum occurs at
\[
y_0 = \mu_0^2 + \frac{1}{2}f' - \mu_0 + \frac{1}{4}.
\]
We will show that $y_0$ lies before our range for $y$ and therefore, the minimum in our range occurs for $y = \mu^2 + f' - \mu_0$. We can compute
\begin{align*}
\psi(\mu_0, f') &:= \frac{2\mu_0^4 + 3f'^2\mu_0 - 4f'\mu_0^2 + \mu_0^3 - 4f'^2 + 5f\mu_0 - \mu_0^2}{2(\mu_0^2 - f' + \mu_0)} - y_0 \\ 
&= \frac{6f'^2\mu_0 - 6f'\mu_0^2 + 2\mu_0^3 - 6f'^2 + 4f'\mu_0 + \mu_0^2 + f' - \mu_0}{4(\mu_0^2 - f' + \mu_0)}.
\end{align*}
The numerator is a parabola in $f'$ with minimum at
\[
f'_0 = \frac{6\mu_0^2 - 4\mu_0 - 1}{12\mu_0 - 12}
\]
and
\[
\psi(\mu_0, f'_0) = \frac{12\mu_0^4 + 24\mu_0^3 - 52\mu_0^2 + 16\mu_0 - 1}{96\mu_0^3 - 48\mu_0^2 - 64\mu_0 + 1} > 0.
\]
Thus, $\varphi(y, f', \mu_0)$ has a minimum $y = \mu^2 + f' - \mu_0$, where
\[
\varphi(\mu^2 + f' - \mu_0, f', \mu_0) = \frac{1}{2}\mu_0 - \frac{1}{2} f' - \frac{5}{24} > 0.
\]
\end{enumerate}
\item Assume $H\cdot\ch(E) = (2, -2\mu_0 - 1, y, z)$.
\begin{enumerate}[(a)]
\item This part is a straightforward calculation.
\item By Theorem \ref{thm:rank_2_first_cases} applied to $E(\mu_0)$, we have $y \leq \mu_0^2 + \mu_0$. The semicircular wall $W$ has center
\[
s(E, \II_C) = - \frac{2\mu_0^2 - 2f' + 2\mu_0 + y}{2\mu_0 + 1},
\]
and the open semidisk $Q_{\alpha, \beta}(\II_C) < 0$ has center
\[
s_Q = -\frac{3e}{2(\mu_0^2 + \mu_0 - f')}.
\]
Since this semidisk's center increases when $e$ decreases, the semidisk itself becomes smaller in that case, and we can bound it from below by setting $e = \tilde{E}(d, \mu_0)$. The lower bound on $y$ is a consequence of the fact that $W$ is outside this semidisk, i.e., $s(E, \II_C) \leq s_Q$.
\item This inequality on $e$ is a direct application of Proposition \ref{prop:maximum_ch3_rank1} to the quotient $G$.
\item 
The upper bound on $z$ follows from Theorem \ref{thm:rank_2_first_cases} applied to $E(\mu_0)$. The bound on $e$ is a direct consequence of estimating $z$ in the previous upper bound for $e$. We have to maximize
\begin{align*}
\varphi(y, f', \mu_0) = &\ \mu_0^4 + f'\mu_0^2 + 2\mu_0^3 - 2\mu_0^2y + \frac{1}{2}f'^2 - f'\mu_0 \\ & + \frac{5}{2}\mu_0^2 - f'y - \mu_0y + y^2 + \mu_0 - y + \frac{1}{4}.
\end{align*}
This function is a parabola in $y$ with minimum at
\[
y_0 = \mu_0^2 + \frac{1}{2}f' + \frac{1}{2}\mu_0 + \frac{1}{2}.
\]
We will show that $y_0$ lies before our range for $y$, and therefore, the maximum occurs at $y =\mu_0^2 + \mu_0$. We want to show that
\begin{align*}
\psi(\mu_0, f') &:= \frac{4\mu_0^4 + 6f'^2\mu_0 - 8f'\mu_0^2 + 8\mu_0^3 - 5f'^2 - 2f'\mu_0 + 7\mu_0^2 - 3f' + 3\mu_0}{4(\mu_0^2 - f' + \mu_0)} - y_0 \\
&= \frac{6f'^2\mu_0 - 6f'\mu_0^2 + 2\mu_0^3 - 3f'^2 - 2f'\mu_0 + 3\mu_0^2 - f' + \mu_0}{4(\mu_0^2 - f' + \mu_0)}
\end{align*}
is non-negative. The numerator is a parabola in $f'$ with minimum at
\[
f'_0 = \frac{6\mu_0^2 + 2\mu_0 + 1}{12\mu_0 - 6}
\]
and
\[
\psi(\mu_0, f'_0) = \frac{12\mu_0^4 + 24\mu_0^3 - 28\mu_0^2 - 16\mu_0 - 1}{8(12\mu_0^3 - 8\mu_0 - 1)} > 0.
\]
Finally,
\[
\varphi\left(\mu_0^2 + \mu_0, f', \mu_0\right) = \mu_0^3 + \frac{1}{2}f'^2 - 2f'\mu_0 + \frac{3}{2}\mu_0^2 + \frac{1}{4} \leq \tilde{E}(d, \mu_0). \qedhere
\]
\end{enumerate}
\end{enumerate}
\end{proof}

\begin{proof}[Proof of Theorem \ref{thm:genus_bound_large_degree}]
By Lemma \ref{lem:complicated_bounds_rank1} part (\ref{item:maximal_rank_of_walls}) we know that $\II_C$ is destabilized by an exact sequence $0 \to E \to \II_C \to G \to 0$, where $H\cdot\ch(E) = (r, x, y, z)$, $E$ is reflexive, and $r = 1$ or $r = 2$.

Assume $r = 1$. Then $E$ is a line bundle with $H\cdot\ch_1(E)=x$ and by Lemma \ref{lem:complicated_bounds_rank1} part (\ref{item:bounding_ch1}) we get $x \leq -k$ and $d > -x(-x-1)$. Then $H\cdot\ch(G) = \left(0, -x, -d - \frac{x^2}{2}, e - \frac{x^3}{6}\right)$, and a direct application of Theorem \ref{thm:rank_zero_bound} gives $e \leq \tilde{E}(d, -x)$. By Lemma \ref{lem:genus_bound_decreases} we get the contradiction $e \leq \tilde{E}(d, k)$.

Assume $r = 2$. As previously, let $\mu_0 \geq k$ be the unique integer such that either $d = \mu_0^2 - f'$ or $d = \mu_0(\mu_0+1) - f'$, where $d \equiv -f' (\mod \mu_0)$ and $0 \leq f' < \mu_0$. Then Lemma \ref{lem:complicated_bounds_rank1} part (\ref{item:bounding_ch1}) implies that we are in any of the four situations described in Lemma \ref{lem:rank2_wall_1} and Lemma \ref{lem:rank2_wall_2}. These lemmas imply that $e \leq \tilde{E}(d, \mu_0)$. Since $\mu_0 \geq k$, we get again a contradiction by using Lemma \ref{lem:genus_bound_decreases} to obtain $e \leq \tilde{E}(d, k)$.
\end{proof}

\section{Towards the Hartshorne-Hirschowitz Conjecture}

In the previous section we gave a new proof for the maximal genus of a degree $d$ curve $C \subset \P^3$ with $H^0(\II_C(k-1)) = 0$ for some positive integer $k$ satisfying $d > k(k-1)$. The main open question is what happens for
\[
\frac{1}{3}(k^2 + 4k + 6) \leq d \leq k(k-1).
\]
Let us rewrite the conjectural answer for this case in terms of the Chern character. Recall from the introduction that for any integer $c \in \Z$, we defined
\[
\delta (c) =  \begin{cases}
3 &\text{if $c=1,3$}\\
1 &\text{if $c \equiv 2 \ (\mod 3)$} \\
0 &\text{otherwise}.
\end{cases}
\]
Then, for any integers $k \geq 5$ and $f \in [k - 1, 2k - 5]$, we defined integers
\begin{align*}
A(k,f) &= \frac{1}{3}(k^2 - kf + f^2 - 2k + 7f + 12 + \delta(2k - f - 6)), \\
B(k,f) &= \frac{1}{3}(k^2 - kf + f^2  + 6f + 11 + \delta(2k - f - 7)).
\end{align*}

\begin{conj}[Hartshorne--Hirschowitz]
\label{conj:hartshorne_rangeB}
Let $C \subset \P^3$ be an integral curve of degree $d$ such that $H^0(\II_C(k-1)) = 0$ for some positive integer $k$. Assume that $A(k,f) \leq d < A(k,f+1)$ for $f \in [k-1, 2k-6]$.
Then
\[
\ch_3(\II_C) \leq E(d,k) := d(k+1) - \binom{k+2}{3} + \binom{f-k+4}{3} + h(d),
\]
where
\[
h(d) = \begin{cases}
0 &\text{if } A(k,f) \leq d \leq B(k,f) \\
\frac{1}{2}(d-B(k,f))(d-B(k,f) + 1) &\text{if } B(k,f) \leq d < A(k,f+1)).
\end{cases}
\]
\end{conj}

The goal of this section is to prove the following result.

\begin{thm}
\label{thm:largeWallRangeB}
Assume the hypothesis of Conjecture \ref{conj:hartshorne_rangeB}. Furthermore, let $A(k, f) \leq d \leq B(k,f)$, and assume that the base field has characteristic $0$. If $\II_C$ is destabilized in tilt stability above or at the numerical wall $W(\II_C, \OO(-f-4)[1])$, then $\ch_3(\II_C) \leq E(d,k)$.
\end{thm}

\subsection{Bounding sections of ideal sheaves}

The results in this section, are elementary ingredients in the proof of Theorem \ref{thm:largeWallRangeB}.

\begin{prop}
\label{prop:sections_points_p2}
Let $Z \subset \P^2$ be a zero-dimensional subscheme of length $n$. Then
\[
h^0(\II_Z(l)) \leq \binom{l+1}{2}
\]
for any integer $l < n$.
\end{prop}

\begin{proof}
If $n = l + 1$, then $\OO_{\P^2}(l)$ is $l$-very ample (see \cite{CG90:d_very_ample} for details on this notion) and
\[
h^0(\II_Z(l)) = \binom{l+1}{2}. 
\]
More generally, for $n > l + 1$ the statement follows from a straightforward induction on $n$.
\end{proof}


The following corollary will be crucial to the proof of the proposition.

\begin{cor}
\label{cor:sections_ideal_bound}
Let $C \subset \P^3$ be an arbitrary one-dimensional subscheme of degree $d$. Then
\[
h^0(\II_C(l)) \leq \binom{l+2}{3}
\]
for any integer $l < d$.
\end{cor}

\begin{proof}
The proof will be achieved by induction on $l$. Indeed, for $l \leq 0$ we have $h^0(\II_C(l)) = 0$. Assume we know the statement holds for some $l > 0$. Let $H \subset \P^3$ be a general plane such that the scheme-theoretic intersection $Z = H \cap C$ is zero-dimensional and of length $d$. We have a short exact sequence
\[
0 \to \OO(l) \to \OO(l+1) \to \OO_H(l+1) \to 0.
\]
By tensoring with $\II_C$ we get another short exact sequence
\[
0 \to \II_C(l) \to \II_C(l+1) \to \II_C \otimes \OO_H(l+1) = \II_{Z/\P^2}(l+1) \to 0.
\]
This sequence is exact on the left, since the map $\II_C(l) \to \II_C(l+1)$ is injective by direct inspection. As a consequence we can use Proposition \ref{prop:sections_points_p2} and the inductive hypothesis to obtain
\begin{align*}
h^0(\II_C(l+1)) &\leq h^0(\II_C(l)) + h^0(\II_{Z/\P^2}(l+1)) \\
&\leq \binom{l+2}{3}
+ \binom{l+2}{2} 
= \binom{l+3}{3} 
. \qedhere
\end{align*}
\end{proof}

\subsection{Proof of Theorem \ref{thm:largeWallRangeB}}

\begin{lem}
\label{lem:f_minus_four_wall}
The objects $\II_C$ and $\OO(-f-4)[1]$ are in the category $\Coh^{\beta}(\P^3)$ along the wall $W(\II_C, \OO(-f-4)[1])$. If $\II_C$ is destabilized by a map $\II_C \to \OO(-f-4)[1]$, then Conjecture \ref{conj:hartshorne_rangeB} holds for $C$.
\end{lem}

\begin{proof}
The first claim holds if and only if
\[
-f-4 < \overline{\beta}(\II_C) = -\sqrt{2d}.
\]
Therefore, we have to show
\[
\frac{(f+4)^2}{2} - d > 0.
\]
By definition
\begin{align*}
\frac{(f+4)^2}{2} - d &> \frac{(f+4)^2}{2} - A(k, f + 1) \\
&\geq \frac{(f+4)^2}{2} - \frac{1}{3}(k^2 - k(f+1) + (f+1)^2 - 2k + 7(f+1) + 15) \\
&= \frac{1}{6}f^2 + \frac{1}{3}fk - \frac{1}{3}k^2 + f + k + \frac{1}{3}.
\end{align*}
The last term is a parabola in $f$ with minimum at $f = - k - 3$. Since we have $f \geq k - 1$, we can get a lower bound by setting $f = k - 1$, where
\[
\frac{1}{6}f^2 + \frac{1}{3}fk - \frac{1}{3}k^2 + f + k + \frac{1}{3} = \frac{1}{6}k^2 + \frac{4}{3}k - \frac{1}{2} > 0.
\]
Let $E$ be the kernel of the morphism $\II_C \to \OO(-f-4)[1]$. The second claim immediately follows from Theorem \ref{thm:hartshorne_reflexive_bounds} applied to $E(k-1)$.
\end{proof}

\begin{lem}
\label{lem:rangeB_rank_bound}
Assume $A(k,f) \leq d \leq B(k,f)$. Then walls for objects with Chern character $\ch(\II_C)$ above or at the numerical wall with $\OO(-f-4)[1]$ are induced by a rank two subobject that is a reflexive sheaf.
\end{lem}

\begin{proof}
Since $C$ is integral, we already showed in Lemma \ref{lem:sub_is_reflexive} that the destabilizing subobject has to be reflexive. All we have to show is that it is of rank two. The radius of $W(\II_C, \OO(-f-4)[1]$ is given by
\[
\rho^2 = \rho(\II_C, \OO(-f-4)[1])^2 = \left(\frac{f+4}{2} - \frac{d}{f+4}\right)^2.
\]
By Lemma \ref{lem:higherRankBound} showing $\rho^2 > \tfrac{d}{12}$ will imply that the subobject has rank at most two. This inequality is equivalent to
\[
\left(\frac{3(f+4)}{4} - \frac{d}{f+4} \right) \left(\frac{f + 4}{3} - \frac{d}{f+4}\right) > 0.
\]
This would follow from 
\[
d < \frac{(f+4)^2}{3}.
\]
The fact  $k - 1 \leq f \leq 2k - 6$ implies $\tfrac{f}{2} + 3 \leq k \leq f + 1$. We know
\[
d \leq A(k, f + 1) - 1 \leq \frac{f^2 - fk + k^2 + 9f - 3k + 20}{3}.
\]
The right hand side is a parabola in $k$ with minimum at $k = \tfrac{f}{2} + \tfrac{3}{2}$. Therefore, we get an upper bound by setting $k = f + 1$ that leads to
\[
d \leq \frac{f^2 + 7f + 18}{3} = \frac{(f+4)^2}{3} - \frac{f-2}{3} < \frac{(f+4)^2}{3}.
\]
Lastly, we have to rule out that $\II_C$ is destabilized by a line bundle. The largest point of the intersection of $W(\II_C, \OO(-f-4)[1]$ with the $\beta$-axis is given by
\[
\beta_0 = -\frac{2d}{f+4}.
\]
We are done, if we can show $\beta_0 > -k$. This is equivalent to showing $d < \tfrac{(f+4)k}{2}$.
We can compute
\begin{align*}
\frac{(f+4)k}{2} - d &\geq \frac{(f+4)k}{2} - B(k,f) \\
&\geq \frac{(f+4)k}{2} - \frac{k^2 - fk + f^2 + 6f + 14}{3} \\
&= -\frac{f^2}{3} + \frac{5fk}{6} - \frac{k^2}{3} - 2f + 2k - \frac{14}{3}.
\end{align*}
The last term defines a parabola in $f$ with maximum. Its minimum has to be given at either $f = k - 1$ or $f = 2k - 6$. For $f = 2k - 6$ we get
\[
\frac{(f+4)k}{2} - d \geq k - \frac{14}{3} > 0,
\]
and for $f = k - 1$, we get 
\[
\frac{(f+4)k}{2} - d \geq \frac{k^2}{6} - \frac{k}{6} - 3 > 0. \qedhere
\]
\end{proof}

\begin{lem}
\label{lem:rangeB_large_ch1}
Assume that $\II_C$ is destabilized at or above the numerical wall with $\OO(-f-4)[1]$, but not by a quotient $\II_C \onto \OO(-f-4)[1]$. 
Furthermore, let $A(k,f) \leq d \leq B(k,f)$. Then $\II_C$ is destabilized by a quotient $\II_C \onto G(-f-5)$ of rank $-1$ with $\ch_1(G) = 0$. Moreover,
\[
\ch_2(G) \in \left[ A(k, f+1) - d, \frac{f+5}{2} - \frac{d}{f+4} \right].
\]
\end{lem}

\begin{proof}
By Lemma \ref{lem:rangeB_rank_bound} we know that $\II_C$ has to be destabilized by an exact sequence $0 \to E \to \II_C \to G' \to 0$ in $\Coh^{\beta}(\P^3)$, where $E$ has rank two and is a reflexive sheaf. Let $x = \ch_1(G')$, $G = G'(x)$, and $y = \ch_2(G)$. Note that by definition $\ch_1(G) = 0$. The wall $W(\II_C, \OO(-f-4)[1])$ intersects the $\beta$-axis at the two points
\begin{align*}
\beta_0 &= -f-4, \\
\beta_1 &= -\frac{2d}{f+4}.
\end{align*}
Since $W(\II_C, \OO(-f-4)[1])$ is smaller than or equal to $W(\II_C, G)$, we get
\[
0 \leq \ch_1^{\beta}(G(-x)) = x + \beta \leq \ch_1^{\beta}(\II_C) = -\beta
\]
for any $\beta$ such that there is a point $(\alpha, \beta) \in W(\II_C, \OO(-f-4)[1])$. In particular, we can choose both $\beta_0$ and $\beta_1$ in these inequalities to obtain
\[
f+4 \leq x \leq \frac{4d}{f+4}.
\]
The center of $W(\II_C, \OO(-f-4)[1])$ is given by
\[
s(\II_C, \OO(-f-4)[1]) = -\frac{f+4}{2} - \frac{d}{f+4}.
\]
The center of $W(\II_C, G)$ is given by
\[
s(\II_C, G) = \frac{y-d}{x} - \frac{x}{2}.
\]
Therefore, the fact that $W(\II_C, \OO(-f-4)[1])$ is smaller than $W(\II_C, G)$ implies
\[
y \leq \frac{(2d - x(f+4))(f+4 - x)}{2(f+4)} =: \varphi_{d,f}(x).
\]
We have
\[
\ch_{\leq 2}(E(k-1)) = \left(2,2k - x - 2, k^2 - kx + \frac{x^2}{2} - d - 2k + x - y + 1\right).
\]
Since $H^0(\II_C(k-1)) = 0$, we must have $H^0(E(k-1)) = 0$. Thus, we can apply Theorem \ref{thm:hartshorne_reflexive_bounds} to get the following bound on $y$
\begin{align*}
y &\geq \frac{k^2}{3} - \frac{kx}{3} + \frac{x^2}{3} - d + \frac{2k}{3} - \frac{x}{3} + \frac{\delta(2k - x - 2)}{3} = A(k, x - 4) - d \\
&\geq \frac{k^2}{3} - \frac{kx}{3} + \frac{x^2}{3} - d + \frac{2k}{3} - \frac{x}{3} =: \psi_{d,k}(x).
\end{align*}
We will rule out specific values of $x$ by showing $\varphi_{d,f}(x) < \psi_{d,k}(x)$. The first step in the proof is to show that if $\varphi_{d,f}(x) < \psi_{d,k}(x)$, then $\varphi_{d,f}(x') < \psi_{d,k}(x')$ for any $x' \geq x$. This will be achieved by comparing the derivatives in $x$ that are given by
\begin{align*}
\varphi'_{d,f}(x) &= x - \frac{d}{f+4} - \frac{f+4}{2},\\
\psi'_{d,k}(x) &= \frac{2x}{3} - \frac{k}{3} - \frac{1}{3}.
\end{align*}
Together with $d \leq A(k,f + 1) - 1$ and $x \leq \tfrac{4d}{f+4}$ we get
\begin{align*}
\psi'_{d,k}(x) - \varphi'_{d,f}(x) &= \frac{f}{2} - \frac{k}{3} - \frac{x}{3} + \frac{5}{3} + \frac{d}{f+4} \\
&\geq \frac{f}{2} - \frac{k}{3} + \frac{5}{3} - \frac{d}{3(f+4)} \\
&\geq \frac{7f^2 - 4fk - 2k^2 + 48f - 18k + 80}{18(f+4)}
\end{align*}
The numerator is a parabola in $f$ with minimum at $f = \tfrac{2}{7}k - \tfrac{24}{7}$. Therefore, it is enough to plug $f = k - 1$ into the numerator, where indeed
\[
k^2 + 20k + 39 \geq 0.
\]
To summarize, we showed that it is enough to rule out $x = f + 6$ by using $A(k,f) \leq d \leq B(k, f)$. In this case the derivative of $\varphi$ by $d$ is larger than the derivative of $\psi$ by $d$. Therefore, we obtain $\varphi_{d,f}(f+6) < \psi_{d,k}(f+6)$, if it holds for any upper bound on $d$, e.g.
\[
d \leq B(k,f) \leq \frac{k^2 - kf + f^2 + 6f + 14}{3}.
\]
Using this upper bound for $d$, we get 
\[
\psi_{d,k}(f+6) - \varphi_{d,f}(f+6) \geq \frac{2}{3} \cdot \frac{2f^2 - 3fk + k^2 + 9f - 8k + 10}{f+4}.
\]
The numerator is a parabola in $f$ with minimum at $f = \tfrac{3}{4}k - \tfrac{9}{4}$. Setting $f = k - 1$ leads to
\[
2f^2 - 3fk + k^2 + 9f - 8k + 10 \geq 3 > 0.
\]
Finally, the bounds on $y$ follow by setting $x = f + 5$ in $\varphi$ and $A(k, x - 4) - d$.
\end{proof}

\begin{lem}
\label{lem:sub_h2_vanishing}
Assume that $\II_C$ is destabilized at or above the numerical wall with $\OO(-f-4)[1]$, by an exact sequence $0 \to E \to \II_C \to G(-f-5) \to 0$, where $\ch_0(G) = -1$, and $\ch_1(G) = 0$. If $A(k,f) \leq d \leq B(k,f)$, then $h^2(E(k-1)) = 0$.
\end{lem}

\begin{proof}
Let $y = \ch_2(G)$. We have
\[
\ch_{\leq 2}(E(k-1)) = \left(2, 2k - f - 7, \frac{1}{2}f^2 - fk + k^2 - d + 6f - 7k - y + \frac{37}{2} \right).
\]
By Theorem \ref{thm:hartshorne_reflexive_bounds} it is enough to show that
\begin{align*}
&\ch_2(E(k-1)) - \left(\frac{1}{6}\ch_1(E(k-1))^2 - \ch_1(E(k-1)) - \frac{8}{3} - \frac{\delta(2k - f - 8)}{3} \right) \\
&= \frac{1}{3}f^2 - \frac{1}{3}fk + \frac{1}{3}k^2 - d + \frac{8}{3}f - \frac{1}{3}k + \frac{\delta(2k - f - 8)}{3} - y + 6
\end{align*}
is non-negative. By Lemma \ref{lem:rangeB_large_ch1} we have
\[
y \leq \frac{f+5}{2} - \frac{d}{f+4}.
\]
Additionally, we can use $d \leq B(k, f)$ and $f \geq k - 1$ to obtain
\begin{align*}
&\frac{1}{3}f^2 - \frac{1}{3}fk + \frac{1}{3}k^2 - d + \frac{8}{3}f - \frac{1}{3}k + \frac{\delta(2k - f - 8)}{3} - y + 6 \\
&\geq \frac{3f^2 - 4fk + 2k^2 + 15f - 8k + 18 + (2f + 8)\delta(2k - f - 8) - (2f + 6)\delta(2k - f - 7)}{6f + 24} \\
&\geq \frac{3f^2 - 4fk + 2k^2 + 9f - 8k}{6f + 24} \geq 0. \qedhere
\end{align*}
\end{proof}

Recall that we denote for any $E \in \Db(\P^3)$ we defined the derived dual $\D(E) = \RHom(E, \OO)[1]$.

\begin{lem}
\label{lem:derived_dual_rank_minusone}
Let $E \in \Coh^{\beta}(\P^3)$ for some $\beta > 0$ be a tilt semistable object for $\alpha \gg 0$ with $\ch(E) = (-1,0,d,e)$. Then there is a distinguished triangle
\[
\II_{C'} \to \D(E) \to T[-1] \to \II_{C'}[1],
\]
where $C' \subset \P^3$ is a closed subscheme, and $T$ is a sheaf supported in dimension zero. If $d > 0$, then $\dim C' = 1$. If $d = 0$, then $\dim C' = 0$. 
\end{lem}

\begin{proof}
Choose $\alpha \gg 0$ such that both $(\alpha, \beta)$ and $(\alpha, -\beta)$ are above the largest wall in tilt stability for walls with respect to $(-1,0,d)$ or $(1,0,-d)$. By Proposition \ref{prop:tilt_derived_dual} we have a distinguished triangle
\[
\tilde{E} \to \D(E) \to T[-1] \to \tilde{E}[1],
\]
where $T$ is a sheaf supported in dimension zero, and $\tilde{E}$ is $\nu_{\alpha, -\beta}$-semistable. By our choice of $\alpha$ and the fact that $\ch_{\leq 2}(\tilde{E}) = (1,0,-d)$ we can use Lemma \ref{lem:large_volume_limit} to see that $\tilde{E}$ is a slope semistable sheaf and thus, must be an ideal sheaf as claimed.
\end{proof}

\begin{proof}[Proof of Theorem \ref{thm:largeWallRangeB}]
Lemma \ref{lem:f_minus_four_wall} already shows that Conjecture \ref{conj:hartshorne_rangeB} holds for $C$, if $\II_C$ is destabilized by a quotient $\II_C \onto \OO(-f-4)[1]$. Therefore, by Lemma \ref{lem:rangeB_large_ch1} and Lemma \ref{lem:sub_h2_vanishing} we can assume that $\II_C$ is destabilized by an exact sequence $0 \to E \to \II_C \to G(-f-5) \to 0$, where $\ch_0(G) = - 1$, $\ch_1(G) = 0$, and $h^2(E(k-1)) = 0$. We denote $y = \ch_2(G)$.

Next we will check that $G(-f-5)$ is stable for $\alpha \gg 0$ by showing that $W(\II_C, G(-f-5))$ is above or equal to the largest wall for $G(-f-5)$. We have $\ch^{-f-4}_1(G(-f-5)) = 1$. Therefore, there is no wall for $G(-f-5)$ that intersects the vertical line $\beta = -f-4$. However, by assumption $W(\II_C, G(-f-5))$ is larger than $W(\II_C, O(-f-4))$ which contains the point $\alpha = 0$, $\beta = -f-4$.

Note that the conjecture is equivalent to
\[
\chi(\II_C(k-1)) \leq \binom{f-k+4}{3}.
\]
The fact $h^0(\II_C(k-1)) = 0$ implies $h^0(E(k-1)) = 0$. Thus, we get
\[
\chi(E(k-1)) \leq -h^1(E(k-1)) \leq -h^0(G(k-f-6)).
\]
By Lemma \ref{lem:derived_dual_rank_minusone} we know that $\D(G)$ fits into a distinguished triangle
\[
\II_{C'} \to \D(G) \to T[-1] \to \II_{C'}[1],
\]
where $C' \subset \P^3$ is a one-dimensional subscheme and $T$ is a sheaf supported in dimension zero. Therefore, we get
\begin{align*}
\chi(\II_C(k-1)) &= \chi(E(k-1)) + \chi(G(k-f-6)) \leq -h^0(G(k-f-6)) + \chi(G(k-f-6)) \\
&\leq h^2(G(k-f-6)) = \ext^2(\D(G), \OO(k-f-6)[1]) \\
&= h^0(\D(G)(f-k+2)) = h^0(\II_{C'}(f-k+2)). 
\end{align*}
The degree of $C'$ is given by
\[
y \geq A(k, f+1) - d \geq A(k, f+1) - B(k,f) = f-k+3.
\]
Henceforth, we can use Corollary \ref{cor:sections_ideal_bound} to obtain
\[
\chi(\II_C(k-1)) \leq h^0(\II_{C'}(f-k+2)) \leq \binom{f-k+4}{3}. \qedhere
\]
\end{proof}

\subsection{An example}

We finish the article by giving a proof of Conjecture \ref{conj:hartshorne_rangeB} for $d = A(k, 2k - 11)$ and $k \geq 31$.

\begin{prop}
\label{prop:special_case_2k10}
Let $C \subset \P^3$ be an integral curve of degree $d = A(k, 2k - 11)$ such that $H^0(\II_C(k-1)) = 0$ for some integer $k \geq 31$ and let $\ch(\II_C) = (1,0,-d,e)$. Then
\[
e \leq E(d, k).
\]
\end{prop}

\begin{lem}
\label{lem:comp_case_2k10}
Assume the hypothesis of Proposition \ref{prop:special_case_2k10}, and assume that $e > E(d, k)$.
\begin{enumerate}[(i)]
\item We have
\begin{align*}
A(k, 2k - 11) &= k^2 - 7k + 19, \\
E(k^2 - 7k + 19, k) &= k^3 - \frac{21}{2}k^2 + \frac{87}{2}k - 65.
\end{align*}
\item The ideal sheaf $\II_C$ is destabilized via a short exact sequence $0 \to E \to \II_C \to G(7 - 2k) \to 0$, where $\ch_1(G) = 0$, and $E$ is a reflexive sheaf of rank two.
\item If $\ch(G) = (-1, 0, y, z)$, then
\[
0 \leq y \leq 6.
\]
\item We have
\[
\chi(G(6-k)) \leq \frac{1}{6}k^3 - 4k^2 - ky + \frac{1}{2}y^2 + \frac{191}{6}k + \frac{17}{2}y - 84.
\]
\item We have $h^2(E(k-1)) \leq \binom{y - 2}{2}$.
\end{enumerate}
\end{lem}

\begin{proof}
\begin{enumerate}[(i)]
\item This part is a simple computation.
\item Let $\rho_Q$ be the radius of the semidisk $Q_{\alpha, \beta}(\II_C) \leq 0$. Using $e > k^3 - \tfrac{21}{2}k^2 + \tfrac{87}{2}k - 65$, we get
\[
\rho^2_Q - \frac{\Delta(\II_C)}{24} > \frac{8k^6 - 168k^5 + 903k^4 + 1402k^3 - 35817k^2 + 147360k - 229600}{48(k^2 - 7k + 19)^2} > 0.
\]
By Lemma \ref{lem:higherRankBound} any destabilizing subobject of $\II_C$ has rank at most two.  Moreover, this shows that the region where $Q_{\alpha, \beta}(\II_C) < 0$ is non-empty. Thus, $\II_C$ has to be destabilized at some point, and by Lemma \ref{lem:sub_is_reflexive} any destabilizing subobject $E$ has to be a reflexive sheaf.

Assume $\ch_0(E) = 1$. Then $E$ is a line bundle. Using $e > k^3 - \tfrac{21}{2}k^2 + \tfrac{87}{2}k - 65$ we obtain
\[
Q_{0, -k}(\II_C) < -7k^3 + 125k^2 - 674k + 1444 < 0.
\]
Using $H^0(\II_C(k-1)) = 0$, we see there are no lines bundles destabilizing $\II_C$ as subobjects. We showed that $E$ is a rank two reflexive sheaf. Let $G(-x)$ be the quotient, where $\ch_1(G) = 0$ for some $x \in \Z$.

Next, we check $x = 2k - 7$. We know the wall $W(E, \II_C)$ has to be outside the region $Q_{\alpha, \beta}(\II_C) < 0$. Thus, for any $\beta \in \R$ with $Q_{0, \beta}(\II_C) < 0$, we get 
\[
0 < \ch_1^{\beta}(G) = x + \beta < \ch_1^{\beta}(\II_C) = -\beta.
\]
The inequality $e > k^3 - \tfrac{21}{2}k^2 + \tfrac{87}{2}k - 65$ can be used to show  
\begin{align*}
Q_{0, -2k + 8}(\II_C) &< -2k^3 + 50k^2 - 308k + 756 \leq 0, \\
Q_{0, -k + 3}(\II_C) &< -k^3 + 38k^2 - 245k + 616 \leq 0.
\end{align*}
In particular, $x = 2k - 7$ holds. Let $s_Q$ be the center of the semidisk $Q_{\alpha, \beta}(\II_C) \leq 0$ on the $\beta$-axis. The fact $s(E, \II_C) \leq s_Q$ together with $e > k^3 - \tfrac{21}{2}k^2 + \tfrac{87}{2}k - 65$ implies
\[
0 \leq y < \frac{9(3k^2 - 23k + 64)}{4(k^2 - 7k + 19)} < 7. \qedhere
\]
\item By Proposition \ref{prop:maximum_ch3_rank1}, we have $z \leq \frac{y(y+1)}{2}$. The statement then follows from the Hirzebruch-Riemann-Roch Theorem. Recall that the Todd class of $\P^3$ is given by $(1,2,\tfrac{11}{6},1)$.
\item This is a direct consequence of Theorem \ref{thm:hartshorne_reflexive_bounds} applied to $E(k-1)$.
\end{enumerate}
\end{proof}

\begin{proof}[Proof of Proposition \ref{prop:special_case_2k10}]
Note that for $d = A(k, 2k - 11)$ Conjecture \ref{conj:hartshorne_rangeB} is equivalent to
\[
\chi(\II_C(k-1)) \leq \binom{f - k + 4}{3} + h(d) = \binom{k - 7}{3}.
\]
Assume $e > E(d,k)$, then Lemma \ref{lem:comp_case_2k10} implies that $\II_C$ is destabilized by a quotient $\II_C \onto G(7 - 2k)$, where $\ch_1(G) = 0$. We have $0 \leq y \leq 6$ for $y = \ch_2(G)$. Then 
\[
\chi(\II_C(k-1)) \leq \frac{1}{6}k^3 - 4k^2 - ky + \frac{1}{2}y^2 + \frac{191}{6}k + \frac{17}{2}y - 84 + \binom{y-2}{2} \leq \binom{k - 7}{3}. \qedhere
\]
\end{proof}

\def\cprime{$'$} \def\cprime{$'$}

\end{document}